\documentclass[a4paper,12pt,reqno]{amsart}

\usepackage[margin=1.2in]{geometry}
\usepackage[hyperfootnotes=false]{hyperref}
\usepackage{amsmath,mathtools,amssymb}
\usepackage{enumerate}
\usepackage{tikz}
\usepackage[all]{xy}
\usepackage{verbatim}
\usepackage{setspace}
\usepackage{enumitem}
\usepackage{mathdots}
\usepackage{mathrsfs}
\usetikzlibrary{decorations.markings}

\DeclareMathOperator{\Hom}{Hom}

\DeclareMathOperator{\End}{End}
\DeclareMathOperator{\Orb}{Orb}

\DeclareMathOperator{\im}{Im}
\DeclareMathOperator{\Ker}{Ker}

\DeclareMathOperator{\Mod}{Mod}

\DeclareMathOperator{\Rep}{Rep}
\DeclareMathOperator{\finrep}{fin}
\DeclareMathOperator{\Vect}{Vect}

\DeclareMathOperator{\fin}{fin}
\DeclareMathOperator{\Aut}{Aut}

\DeclareMathOperator{\ev}{ev}
\DeclareMathOperator{\ad}{ad}

\newcommand{\mc}[1]{\mathcal{#1}}
\newcommand{\mb}[1]{\mathbb{#1}}
\newcommand{\mf}[1]{\mathfrak{#1}}

\newcommand{\rs}{\mb{X}}
\newcommand{\poles}{S}
\newcommand{\srs}{\mb{Y}}

\newcommand{\mbX}{\rs}
\newcommand{\id}{\mathrm{id}}
\newcommand{\Rsphere}{\overline{\mathbb{C}}}
\newcommand{\OX}{\mathcal{O}_\mbX}
\newcommand{\sltwo}{\mathfrak{sl}_2(\mathbb{C})}
\newcommand{\otimesC}{\otimes_{\mathbb{C}}}
\newcommand{\generalALA}{\mathfrak{A}(\mathfrak{g},\mbX,\Gamma,\sigma_{\mathfrak{g}},\sigma_\mbX)}
\newcommand{\lieg}{\mathfrak{g}}
\newcommand{\mfA}{\mathfrak{A}}
\newcommand{\rg}{\Gamma}
\newcommand{\basis}{\mathfrak{B}_\lieg}
\newcommand{\genrs}{x}
\newcommand{\genrstwo}{y}
\newcommand{\polyvar}{z}
\newcommand{\jet}[2]{J_{#1,#2}}
\newcommand{\vjet}[2]{\id\otimes J_{#1,#2}}
\newcommand{\jordan}{J}
\newcommand{\indrep}[2]{\varphi_{#1,#2}}
\newcommand{\jetrep}[2]{\phi_{#1,#2}}
\newcommand{\ideal}{\mathfrak{I}}

\newcommand{\strictpos}{\mathbb{Z}_{> 0}}

\theoremstyle{plain}
\newtheorem{thm}{Theorem}[section]
\newtheorem*{thm*}{Theorem}

\newtheorem{lem}[thm]{Lemma}
\newtheorem{cor}[thm]{Corollary}
\newtheorem*{cor*}{Corollary}
\newtheorem{prop}[thm]{Proposition}
\newtheorem*{prop*}{Proposition}

\theoremstyle{definition}
\newtheorem{defn}[thm]{Definition}

\newtheorem{exam}[thm]{Example}
\newtheorem*{exam*}{Example}

\theoremstyle{remark}
\newtheorem{rem}[thm]{Remark}

\numberwithin{equation}{section}

\usepackage{etoolbox}
\usepackage[style=numeric,firstinits=true,url=false]{biblatex}
\usepackage{mathscinet}
\begin{document}
		\title[{Wild Local Structures of ALiAs}]{Wild Local Structures of \\Automorphic Lie Algebras}

\author{Drew Duffield}
\address{Department of Mathematical Sciences,
Durham University,
Lower Mountjoy,
Stockton Road,
Durham DH1 3LE,
UK}
\email{drew.d.duffield@durham.ac.uk }

\author{Vincent Knibbeler}\address{Department of Mathematical Sciences,
Loughborough University, Loughborough LE11 3TU, UK}
\email{V.Knibbeler@lboro.ac.uk}

\author{Sara Lombardo}\address{Department of Mathematical Sciences,
Loughborough University, Loughborough LE11 3TU, UK}
\email{S.Lombardo@lboro.ac.uk}

	\maketitle
	\begin{abstract}
We study automorphic Lie algebras using a family of evaluation maps parametrised by the representations of the associative algebra of functions. This provides a descending chain of ideals for the automorphic Lie algebra which is used to prove that it is of wild representation type. 
We show that the associated quotients of the automorphic Lie algebra are isomorphic to twisted truncated polynomial current algebras. When a simple Lie algebra is used in the construction, this allows us to describe the local Lie structure of the automorphic Lie algebra in terms of affine Kac-Moody algebras.
	\end{abstract}

\vspace{1cm}

\section{Introduction}
Let $\srs$ be a Riemann surface and $\lieg$ a complex finite dimensional Lie algebra, and suppose $\rg$ is a group 
acting effectively on both $\srs$ and $\lieg$ by automorphisms.
The space of meromorphic maps from $\srs$ to $\lieg$ becomes a Lie algebra when we define the bracket of two such maps as the pointwise bracket in $\lieg$. The Lie subalgebra of $\rg$-equivariant maps which are holomorphic on the complement $\rs$ of a $\rg$-invariant subset $\poles$ of $\srs$ is an automorphic Lie algebra. In this paper $\srs$ will be compact and $\rg$ finite. We will moreover assume $\poles$ to be finite and call $\rs$ a punctured compact Riemann surface. The best known examples are the genus zero cases with one and two punctures and with cyclic reduction group: the twisted current algebras\footnote{Some authors denote by current algebras a far greater set of Lie algebras. In this paper a current algebra will always be of the form $\mf{g}\otimes_{\mb{C}}\mb{C}[\polyvar]$.} $(\mf{g}\otimes_{\mb{C}}\mb{C}[\polyvar])^\rg$ and twisted loop algebras $(\mf{g}\otimes_{\mb{C}}\mb{C}[\polyvar,\polyvar^{-1}])^\rg$ respectively.
Defined originally in the context of integrable systems \cite{lombardo,lombardo2004reductions,lombardo2005reduction} where $\rg$ (more precisely its representation by automorphisms of the Lie algebra) is called a \emph{reduction group} \cite{Mikhailov81,lombardo2004reductions,lombardo2005reduction}, automorphic Lie algebras are also known as equivariant map algebras, following a completely independent development, see, e.g. \cite{neher2012irreducible}, these latter being defined in a more general context. Note however that in the original papers on automorphic Lie algebras, $\srs$ was taken to be the Riemann sphere, and thus this paper works with a more general notion of automorphic Lie algebras than those in \cite{lombardo,lombardo2004reductions,lombardo2005reduction}.

The aim of this paper is to lay down foundations for a representation theory of automorphic Lie algebras.
The starting point in representation theory of finite-dimensional algebras is to understand the representation type, namely whether the object one is considering is of \emph{tame} or \emph{wild} representation type. In the former case, indecomposable modules in each dimension occur in a finite number of one-parameter families, while the latter case is much more difficult, at least as complicated as the representation theory of all finite-dimensional algebras. Automorphic Lie algebras are infinite-dimensional, and there the tame-wild dichotomy is less understood (see \cite{iovanov2018} for an account), nonetheless answering the \emph{type question} is of fundamental importance. 

We reach the conclusion that automorphic Lie algebras are of wild representation type by studying the Lie algebra in the vicinity of a point of the Riemann surface and showing that this local Lie algebra is of wild representation type. This approach leads to a second result of this paper: a complete local structure theory of automorphic Lie algebras. In contrast, the previous literature on automorphic Lie algebras has been almost exclusively concerned with global structure theory, with a computational programme on one hand, most recently \cite{knibbeler2017higher}, and a theoretical programme on the other \cite{knibbeler2019hereditary, knibbeler2020cohomology}. Even on the Riemann sphere, the global structure of automorphic Lie algebras is not yet completely understood.

Our first aim is to investigate what role the representation theory of the ring of holomorphic functions on the Riemann surface $\rs$ has on the structure of the automorphic Lie algebra; in Section \ref{sec:ideals}, we show that the representations of the ring of holomorphic functions on the punctured Riemann surface give rise to various ideals and quotients of the automorphic Lie algebra. 
In particular, we show that certain composition series of representations in the category of representations of the space of holomorphic functions on $\rs$ gives rise to a descending chain of ideals of the automorphic Lie algebra, and thus also a series of quotients.
In Section \ref{sec:wildness}, we use this chain of ideals to prove that automorphic Lie algebras are wild. 
In Section \ref{sec:closer} we show that the quotients of the automorphic Lie algebras are isomorphic to twisted truncated current algebras, using a local coordinate on the Riemann surface that linearises the group action together with the Riemann-Roch theorem. 
Section \ref{sec:local} specialises to the case where $\mf{g}$ is a simple Lie algebra, and we use root systems of Kac-Moody algebras to formulate a local structure theory for the automorphic Lie algebras. 

The first quotient in the series essentially classifies finite-dimensional irreducible representations of automorphic Lie algebras \cite{neher2012irreducible}. We hope that the full series will help to provide the foundations for developing the representation theory of automorphic Lie algebras further, and so we end the paper with Section~\ref{sec:directions}, where we present some potential further directions of research.
	
	\section*{Notation}
	For an arbitrary (Lie or associative) algebra $\mathfrak{A}$, we denote by $\Rep \mathfrak{A}$ (resp. $\finrep \mathfrak{A}$) the category of representations (resp. finite-dimensional representations) of $\mathfrak{A}$. We will often write an object in the category $\Rep \mathfrak{A}$ as a pair $(V, \rho)$ where $V$ is a vector space and $\rho\colon \mathfrak{A} \rightarrow \End_{\mathbb{C}}(V)$ is a linear map. For the purposes of readability, we will sometimes refer to a representation $(V, \rho)$ simply as $\rho$, provided the context is clear. 
	
	We denote by $\otimes_\mathbb{C}$ the usual tensor product in the category $\Vect \mathbb{C}$ of $\mathbb{C}$-vector spaces. For finite matrices $X$ and $Y$ over $\mathbb{C}$, $X \otimes_{\mathbb{C}} Y$ is interpreted as the usual Kronecker product.
	
	Throughout, we let $\srs$ be a compact Riemann surface, We let $\lieg$ be a complex finite-dimensional Lie algebra, and we let $\rg$ be a finite group that acts faithfully on both $\lieg$ and $\srs$ by automorphisms. 
Moreover we let $\poles$ be a finite, non-empty, $\rg$-invariant subset of $\srs$ and $\rs$ the punctured compact Riemann surface $\rs=\srs\setminus \poles$.
When it is convenient to have a name for the action of $\rg$ on $\mf{g}$ and $\rs$ we will use
	\begin{align*}
		\sigma_{\lieg}&\colon \rg \times \lieg \rightarrow \lieg,\\
		\sigma_{\mbX}&\colon \rg \times \mbX \rightarrow \mbX
	\end{align*}
with $\sigma_{\lieg}(\gamma, A)=\gamma A$ and $\sigma_{\mbX}(\gamma,\genrs)=\gamma\genrs$.

We denote by $\OX$ the space of meromorphic functions on $\srs$ which are holomorphic on $\rs$ (it is also sensible and common to denote this space by $\mc{O}_\srs(\rs)$ since it depends on the compact Riemann surface $\srs$, but we opt for the shorter notation $\OX$ as there is little opportunity for confusion). 
The action of $\rg$ on $\rs$ then induces an action on $\mathcal{O}_\mbX$ by $\gamma f=f \circ \gamma^{-1}$ for any $\gamma \in \rg$ and any rational function $f \in \mathcal{O}_\mbX$.
	
 We denote by $\mathfrak{A}(\lieg, \mbX, \rg, \sigma_{\lieg}, \sigma_\mbX)$ the automorphic Lie algebra
	\begin{equation*}
		(\lieg \otimes_{\mathbb{C}} \mathcal{O}_\mbX)^\rg = \{ a \in \lieg \otimes_{\mathbb{C}} \mathcal{O}_\mbX : \gamma a=a \text{ for any } \gamma \in \rg\},
	\end{equation*}
	where the action of $\rg$ on $\lieg$ and $\mbX$ is precisely $\sigma_\lieg$ and $\sigma_\mbX$ respectively. Quite often, it is convenient to write the elements of $\generalALA$ as $\sum_{i \in I} A_i \otimes f_i$ for some $A_i \in \lieg$, some $f_i \in \OX$ and some index set $I$.

	\section{Preliminaries} \label{sec:prelim}
	
	\subsection{Point-Evaluation Representations}
	Neher, Savage and Senesi classified the irreducible finite-dimensional representations of equivariant map algebras in \cite{neher2012irreducible}. As a subclass of equivariant map algebras, this classification specialises to automorphic Lie algebras on compact Riemann surfaces.

The classification in \cite{neher2012irreducible} involves certain representations called evaluation representations. We present a brief recollection of the evaluation representations of \cite{neher2012irreducible} in this subsection, which we instead call \emph{point-evaluation representations}.
	
	\begin{defn}[\cite{neher2012irreducible}]
		Let $\mathfrak{A}=\generalALA$ be an automorphic Lie algebra and $X\subset \mbX$ be a finite subset. For each $\genrs \in X$, let $\lieg^\genrs$ denote the Lie subalgebra
		\begin{equation*}
			\lieg^\genrs=\{A \in \lieg : \gamma A = A \text{ for any } \gamma \in \rg_\genrs\} \subseteq \lieg,
		\end{equation*}
		where $\rg_\genrs$ is the stabiliser subgroup of $\rg$ with respect to $\genrs$. The \emph{point-evaluation map} of the set $X$ is the map
		\begin{equation*}
			\ev_X\colon \mathfrak{A} \rightarrow \bigoplus_{\genrs \in X} \lieg^\genrs
		\end{equation*}
		defined such that $\ev_X\left(\sum_{i} A_i \otimes f_i\right)=\left(\sum_{i} f_i(\genrs) A_i\right)_{\genrs \in X}$ for any element $\sum_{i} A_i \otimes f_i \in \mfA$.
	\end{defn}
	
	The point-evaluation map plays an essential role in the representation theory of $\mfA$. Not least because it allows us to define irreducible representations of $\mfA$ in the following way.
	
	\begin{defn}[\cite{neher2012irreducible}] \label{PointRep}
		Let $\mfA=\generalALA$ be an automorphic Lie algebra and let $X\subset \mbX$ be a finite subset such that $\genrs \not\in \Orb(\genrstwo)$ for any two distinct points $\genrs,\genrstwo \in X$. In addition, let $\rho_\genrs \colon \lieg^\genrs \rightarrow \End_{\mathbb{C}}(V_\genrs) \in \fin \lieg^\genrs$ be irreducible for each $\genrs \in X$. A \emph{point-evaluation representation} of $\mfA$ with respect to $X$ is a representation of the form
	\begin{equation*}
		\left(\bigotimes_{\genrs \in X}\rho_\genrs\right) \circ \ev_X \colon \mfA \rightarrow \End_{\mathbb{C}}\left( \bigotimes_{\genrs \in X} V_\genrs \right),
	\end{equation*}
	where the tensor product on the left-hand-side is in $\fin \left(\bigoplus_{\genrs \in X} \lieg^\genrs\right)$.
	\end{defn}
	
	\subsection{Filtrations of Representations}
	We make use of terminology in representation theory typically reserved for modules over rings and associative algebras, but which is easily adaptable to representations of Lie algebras. Let $\mathfrak{A}$ be an arbitrary (Lie or associative) algebra. Since $\finrep \mathfrak{A}$ is an abelian category, one has a \emph{composition series}
	\begin{equation*}
		(V,\rho) = (V_0,\rho_0) \supset (V_1,\rho_1) \supset \ldots \supset (V_n,\rho_n) = 0
	\end{equation*}
	for any representation $(V,\rho) \in \finrep \mathfrak{A}$, where each $(V_i,\rho_i) / (V_{i+1},\rho_{i+1})$ is irreducible. We define the \emph{length} of $(V,\rho)$, which we denote by $\ell(V,\rho)$, to be the least integer $n$ such that $(V_n,\rho_n)=0$. A composition series of a representation $(V,\rho) \in \finrep \mathfrak{A}$ is not necessarily unique, but the length of a representation is well-defined. That is, the integer $n$ depends only on $(V,\rho)$.
	
	\begin{defn}
		Let $\mathcal{A}$ be an abelian category. We say an object $\rho \in \mathcal{A}$ is \emph{uniserial} if it has a unique composition series.
	\end{defn}
	
	\subsection{Wild Representation Type}
	A fundamental goal in representation theory is to understand the category $\fin \mfA$. An important step towards a complete understanding of this category is a complete classification of the indecomposable representations. For many algebras, this is considered to be a wild problem, which means it contains the problem of classifying pairs of commuting matrices up to simultaneous similarity (a precise statement is given in the definitions below). This is considered to be extremely hard --- not least because it contains the problem of classifying the indecomposable representations of all finite-dimensional associative algebras.
	
	\begin{defn}
		Let $\mfA$ be an arbitrary (Lie or associative) algebra over an algebraically closed field $K$. An additive, full, exact subcategory $\mathcal{C} \subseteq \fin \mfA$ is said to be wild if there exists an exact $K$-linear functor $F\colon \fin K[x,y] \rightarrow \mathcal{C}$ that maps indecomposable objects to indecomposable objects and respects isomorphism classes (that is, $F(V,\rho)\cong F(V',\rho')$ implies $(V,\rho) \cong (V',\rho')$).
	\end{defn}
	As mentioned above, this definition is equivalent to the following.
	\begin{prop}[\cite{BlueBookIII}, XIX.1.11]
		An additive, full, exact subcategory $\mathcal{C} \subseteq \fin \mfA$ is wild if and only if for any finite-dimensional associative algebra $\Lambda$, there exists an exact $K$-linear functor $F\colon \fin \Lambda \rightarrow \mathcal{C}$ that maps indecomposable objects to indecomposable objects and respects isomorphism classes.
	\end{prop}
	
	We say that an arbitrary (Lie or associative) algebra $\mfA$ is wild if $\fin \mfA$ is wild. Usually, it is easier to prove wildness by investigating the quotients of an algebra, as $\mfA$ is wild if there exists an ideal $\mathfrak{I} \subseteq \mfA$ such that $\mfA / \mathfrak{I}$ is wild. For finite-dimensional Lie algebras, a classification of wild Lie algebras is known due to Makedonski\u{\i}.
	\begin{thm}[\cite{makedonskii2013on}, Theorem 3] \label{MakedonskiiThm}
		A finite-dimensional Lie algebra over an algebraically closed field is wild if and only if it is neither semisimple, one-dimensional, nor a direct sum of a semisimple Lie algebra and a one-dimensional Lie algebra.
	\end{thm}
	
	Semisimple, one-dimensional Lie algebras, and Lie algebras isomorphic to a direct sum of a semisimple Lie algebra and a one-dimensional Lie algebra are said to be \emph{tame}. Informally, this means that in each dimension $m$, all but a finite number of indecomposable representations occur in one-parameter families. For algebras where one hopes to classify the finite-dimensional indecomposable representations, one would expect the algebra to be tame.
	
	\subsection{The Finitely Generated Representation Theory of $\OX$ when $\srs$ is Genus Zero} \label{MCGammaReps}
	A special case of the ring $\OX$ occurs when $\srs$ is genus zero. In this case, $\srs$ is the Riemann sphere $\Rsphere=\mathbb{C} \cup \{\infty\} \cong \mathbb{CP}^1$ and $\rs$ is a punctured Riemann sphere $\Rsphere \setminus S$ by a finite set of $\rg$-invariant points $S$. It follows that $\OX$ is the commutative associative algebra
	\begin{equation*}
		\left\{\frac{p}{q} \in \mathbb{C}(\polyvar) : \frac{q(\varepsilon)}{p(\varepsilon)} = 0 \Rightarrow \varepsilon \in S\right\}
	\end{equation*}
	whose finitely generated representation theory is tame and completely understood.

	Explicitly, it is easy to see that if $S=\{\infty, \varepsilon_1,\ldots,\varepsilon_n\}$, we have
	\begin{equation*}
		\OX \cong \mathbb{C}[\polyvar, \widehat{\polyvar}_1,\ldots,\widehat{\polyvar}_n],
	\end{equation*}
	where $\widehat{\polyvar}_i=(\polyvar-\varepsilon_i)^{-1}$ for each $1 \leq i \leq n$. If we instead have $S=\{\varepsilon_1,\ldots,\varepsilon_n\}$ such that $\infty \not\in S$, then we simply remove $\polyvar$ from the generating set of the polynomial algebra. That is, 
	\begin{equation*}
		\OX \cong \mathbb{C}[\widehat{\polyvar}_1,\ldots,\widehat{\polyvar}_n].
	\end{equation*}
	
	Any algebra of the second form is isomorphic to an algebra of the first form. One may choose an automorphism $h$ of the Riemann sphere (a M\"{o}bius transformation) and construct a corresponding ring isomorphism $\widehat{h}\colon \OX \rightarrow \mathcal{O}_{h(\rs)}$ defined by $\widehat{h}(f)=f \circ h^{-1}$ for any rational function $f \in \OX$. Specifically, if $h$ is defined such that $h(\genrs)=(\genrs-\varepsilon_i)^{-1}$ for any $\genrs \in \mbX$ and for some $\varepsilon_i \in S$, then $\widehat{h}(\widehat{\polyvar}_i)=\polyvar$.
	
	The tameness of the finitely generated representation theory of $\OX$ arises from the fact that $\OX$ is isomorphic to the localisation of the polynomial algebra $\mathbb{C}[\polyvar]$ at the non-zero multiplicative closed set
	\begin{equation*}
		\widehat{S}=\{(\polyvar-\varepsilon)^d : \varepsilon \in S \setminus \{\infty\}, d \in \mathbb{Z}_{\geq 0}\}
	\end{equation*}
	and is hence a principal ideal domain (PID). Thus, the structure theorem for finitely generated modules over a PID applies. We will provide a brief summary here.
	
	\subsubsection{The indecomposable representations}
	There is only one finitely generated infinite-dimensional indecomposable representation of $\OX$. Namely, this is the representation corresponding to the free module $\OX$. The finite-dimensional indecomposable representations of $\OX$ are in bijective correspondence with the primary ideals of $\OX$. In particular, the $m$-dimensional indecomposable representations correspond to the ideals $I^m$, where $I$ is a prime ideal of $\OX$. Since $\OX$ is a PID, the prime ideals are precisely the maximal ideals, and hence in each dimension $m$, the indecomposable representations are in bijective correspondence with the points of the variety $\mbX$.
	
	The $m$-dimensional indecomposable representations of $\OX$ are (almost) completely determined by Jordan blocks
	\begin{equation*}
		\jordan_{\genrs, m} = 
		\begin{pmatrix}
			\genrs	&	1		&	0		&	\cdots	&	0		\\
			0		&	\genrs	&	1		&	\cdots	&	0		\\
			\vdots	&	\vdots	&	\ddots	&	\ddots	&	\vdots	\\
			0		&	0		&	\cdots	&	\genrs	&	1		\\
			0		&	0		&	\cdots	&	0		&	\genrs
		\end{pmatrix}
	\end{equation*}
	such that $\genrs \not\in S$. Specifically, if $\infty \in S$, then an $m$-dimensional representation $\varphi \in \finrep \OX$ is indecomposable if and only if it is isomorphic to a representation $\indrep{\genrs}{m}$ for some $\genrs \in \mbX$, defined by $\indrep{\genrs}{m}(\polyvar)=\jordan_{\genrs,m}$ and $\indrep{\genrs}{m}(\widehat{\polyvar}_i)=\jordan_{\genrs-\varepsilon_i,m}^{-1}$ for any $i$.
	
	The description of indecomposable representations of $\OX$ in the case where $\infty \not\in S$ is slightly different.
	However, one may again note that there is a ring isomorphism between the $\infty \in S$ and $\infty \not\in S$ cases, and thus the representation theory is equivalent between these cases. For this reason, we will omit an explicit description of the representation theory in the $\infty \not\in S$ case.
	
	An indecomposable representation $\varphi \in \finrep \OX$ is irreducible precisely when $m=1$. That is, whenever $\varphi$ is a 1-dimensional representation. In later examples, we will use particular finite-dimensional indecomposable representations of $\OX$ when $\mbX = \Rsphere \setminus S$. Thus, we have the following definition.
	\begin{defn}
		Let $\rs = \Rsphere \setminus S$. We define the representation $\indrep{\genrs}{m} \in \finrep \OX$ to be the indecomposable representation constructed above.
	\end{defn}
	
	\begin{rem} \label{MatFunction}
		Note that 
		we have chosen $\indrep{\genrs}{m}$ in such a way that 
		\begin{equation*}
			(\indrep{\genrs}{m}((\polyvar-\varepsilon_i)^{-1}))^{-1} + \varepsilon_i\id_m = \jordan_{\genrs,m}.
		\end{equation*}
		Thus in this sense, $\indrep{\genrs}{m}$ respects algebraic manipulation where $\polyvar$ is considered as a placeholder for the matrix $\jordan_{\genrs,m}$. Hence given any $f \in \OX$, we can consider $\indrep{\genrs}{m}(f)$ as a matrix rational function $f(\jordan_{\genrs,m})$ whenever $\genrs \neq \infty$. A well known consequence of this is that
		\begin{equation*}
			\indrep{\genrs}{m}(f) =
			\begin{pmatrix}
				f(\genrs)	& f'(\genrs)	& \frac{f''(\genrs)}{2}	& \cdots		& \frac{f^{(m-1)}(\genrs)}{(m-1)!} \\
				0			& f(\genrs)	& f'(\genrs)				& \cdots		& \frac{f^{(m-2)}(\genrs)}{(m-2)!} \\
				\vdots		& \vdots		& \ddots					& \ddots		& \vdots	\\
				0			& 0			& \cdots					& f(\genrs)	& f'(\genrs)		\\
				0			& 0			& \cdots					& 0			& f(\genrs)		\\
			\end{pmatrix},
		\end{equation*}
		where $f^{(i)}$ is the $i$-th derivative of $f$ with respect to $\polyvar$.
	\end{rem}
	
	\subsubsection{The morphisms of $\fin \OX$}
	Every finite-dimensional indecomposable representation of $\OX$ is uniserial. Namely, we have a unique composition series
	\begin{equation*}
		(V_{\genrs,m}, \indrep{\genrs}{m}) \supset ( V_{\genrs,m-1}, \indrep{\genrs}{m-1}) \supset \ldots \supset (V_{\genrs,1}, \indrep{\genrs}{1}) \supset 0,
	\end{equation*}
	for any indecomposable $(V_{\genrs,m}, \indrep{\genrs}{m}) \in \fin \OX$. In particular, we have 
	\begin{equation*}
		(V_{\genrs,i}, \indrep{\genrs}{i}) / ( V_{\genrs,i-1}, \indrep{\genrs}{i-1}) \cong (V_{\genrs,1}, \indrep{\genrs}{1})
	\end{equation*}
	for all $i$. This induces an infinite ascending chain of subrepresentations
	\begin{equation*}
		0 \subset (V_{\genrs,1}, \indrep{\genrs}{1}) \subset (V_{\genrs,2}, \indrep{\genrs}{2}) \subset \ldots,
	\end{equation*}
	There are irreducible monomorphisms $(V_{\genrs,i}, \indrep{\genrs}{i}) \rightarrow (V_{\genrs,i+1}, \indrep{\genrs}{i+1})$ and irreducible epimorphisms $(V_{\genrs,i}, \indrep{\genrs}{i}) \rightarrow (V_{\genrs,i-1}, \indrep{\genrs}{i-1})$ in $\fin \OX$ for all $i$. In addition, we have 
	\begin{equation*}
		\Hom_{\OX}((V_{\genrs,i}, \indrep{\genrs}{i} ),(V_{\genrs',j}, \indrep{\genrs'}{j}))=0
	\end{equation*}
	for any $\genrs \neq \genrs'$.
	
	\subsection{Representations of $\OX$ Arising from Jet Maps} \label{sec:RepsFromJets}
	Suppose $\srs$ is of arbitrary genus. Since the elements of $\OX$ are holomorphic functions on $\rs$, one can consider a homomorphism of commutative associative algebras
	\begin{equation*}
		\jet{\genrs_0}{m} \colon\OX\rightarrow\mathbb{C}[\polyvar]/(\polyvar^{m+1})
	\end{equation*}
	defined by sending a meromorphic function $f$ to its Taylor expansion about $\genrs_0 \in \rs$ in the local coordinate $\polyvar$ modulo $(\polyvar^{m+1})$. Such a map is called a \emph{jet map}.
	
	By the previous subsection, the category $\fin \mathbb{C}[\polyvar]/(\polyvar^{m+1})$ is completely understood --- it is the full abelian subcategory of $\fin \mathbb{C}[\polyvar]$ consisting of representations $\indrep{0}{r}$ with $r \leq m+1$. One can therefore obtain an $r$-dimensional representation $\jetrep{\genrs}{r}\in\fin \OX$ defined by the following commutative diagram.
	\begin{equation*}
		\xymatrix{
			\OX
				\ar[r]^-{\jet{\genrs}{m}}
				\ar[dr]^-{\jetrep{\genrs}{r}}
			& \mathbb{C}[\polyvar]/(\polyvar^{m+1})
				\ar[d]^-{\indrep{0}{r}}
			\\
			&
			\End_{\mathbb{C}}(V_{0,r})
		}
	\end{equation*}
	Thus for each representation $\indrep{0}{r} \in \fin \mathbb{C}[\polyvar]/(\polyvar^{m+1})$, one obtains a representation $\jetrep{\genrs}{r}$ given by precomposition with $\jet{\genrs}{m}$. We call such a representation a \emph{jet representation} of $\OX$. Similarly, given a morphism
	\begin{equation*}
		\vartheta\colon (V_{0,r}, \indrep{0}{r}) \rightarrow (V_{0,r'}, \indrep{0}{r'}) \in\fin \mathbb{C}[\polyvar],
	\end{equation*}
	one can see that there exists a corresponding morphism
	\begin{equation*}
		\theta\colon  (V_{0,r}, \jetrep{\genrs}{r}) \rightarrow (V_{0,r'}, \jetrep{\genrs}{r'}) \in \fin \OX
	\end{equation*}
	defined by $\theta(v)=\vartheta(v)$ for all $v \in V_{0,r}$. This defines a family of exact functors
	\begin{equation*}
		F_m \colon \fin \mathbb{C}[z]/(z^{m+1}) \rightarrow \fin \OX
	\end{equation*}
	such that $F_{m'}\theta = F_m \theta$ for any
	\begin{equation*}
		\theta \in \Hom_{\OX}((V_{0,r}, \jetrep{\genrs}{r}), (V_{0,r'}, \jetrep{\genrs}{r'}))
	\end{equation*}
	with $r,r' \leq m' \leq m$. Hence for each representation $\jetrep{\genrs}{m}$, one obtains a composition series
	\begin{equation*}
		(V_{0,m}, \jetrep{\genrs}{m}) \supset ( V_{0,m-1}, \jetrep{\genrs}{m-1}) \supset \ldots \supset (V_{0,1}, \jetrep{\genrs}{1}) \supset 0,
	\end{equation*}
	where for each $i$, we have 
	\begin{equation*}
		(V_{0,i}, \jetrep{\genrs}{i}) / ( V_{0,i-1}, \jetrep{\genrs}{i-1}) \cong (V_{0,1}, \jetrep{\genrs}{1}),
	\end{equation*}
	which is necessarily irreducible. This is precisely the image under $F_m$ of the composition series in the previous subsection with $\genrs = 0$.
	
	\begin{rem}
		It is worth noting that whilst the category $\fin \mathbb{C}[\polyvar]/(\polyvar^{m+1})$ is tame (actually, of finite representation-type), we make no such claims for $\fin \OX$.
	\end{rem}
	
	\begin{rem} \label{JetMatFunction}
		Given a function $f \in \OX$, the matrix of the linear map $\jetrep{\genrs}{r}$ is the same as $\indrep{\genrs}{r}$ in Remark~\ref{MatFunction}. This is because we have
		\begin{equation*}
			\indrep{\genrs}{0}\circ\jet{\genrs}{r}(f)= f(\genrs)
			\begin{pmatrix}
				1 		& 0 & \cdots & 0 \\
				0 		& 1 & \ddots &\vdots \\
				\vdots	& \ddots & \ddots & 0 \\
				0		& \cdots & 0 & 1 
			\end{pmatrix}+
			\sum_{i=1}^{r-1} \frac{f^{(i)}(\genrs)}{i!}
			\begin{pmatrix}
				0		& 1 		& 0			& \cdots	& 0 \\
				0		& 0			& 1		 	& \ddots 	& \vdots \\
				0		& 0			& 0			& \ddots	& 0 \\
				\vdots	& \ddots 	& \ddots	& \ddots	& 1 \\
				0		& \cdots	& 0			& 0			&0
			\end{pmatrix}
			^i
		\end{equation*}
	\end{rem}

		\section{Ideals and Quotients of Automorphic Lie Algebras} 
	\label{sec:ideals}	
	Recall that the definition of the Lie bracket of an automorphic Lie algebra $\mfA=(\lieg \otimesC \OX)^\rg$ is given by
	\begin{equation*}
		\left[ \sum_i A_i \otimes f_i, \sum_j B_j \otimes g_j\right]_{\mfA} = \sum_i\sum_j [A_i,B_j]_{\lieg} \otimes f_i g_j.
	\end{equation*}
	An interesting feature of this definition is the associative product on the right-hand-side of the tensor product, which comes from the commutative associative algebra $\OX$. Therefore, it is perhaps natural to wonder precisely what role the representation theory of $\OX$ plays on the structure of $\mfA$. The first few results of this section shows that the representation theory of $\OX$ is closely related to the ideals and quotients of $\mfA$. This applies to any commutative associative algebra $\OX$, and thus these results apply to equivariant map algebras as well as automorphic Lie algebras. 
	
	Throughout this section, we will fix a basis $\basis$ of the Lie algebra $\lieg$ used in the construction of a given automorphic Lie algebra $\mfA$. The first result of this section shows how one can associate to each (not necessarily finite-dimensional) representation $\varphi \in \Rep \OX$ an ideal $\ideal_\varphi$ of $\mfA$ and a corresponding quotient $\mfA / \ideal_\varphi$.
	
	\begin{prop} \label{prop:IdealsFromReps}
		Let $\mfA=\generalALA$ be an automorphic Lie algebra and $\varphi \colon \OX \rightarrow \End_\mathbb{C}(V)$ be a representation of $\OX$. The kernel of the linear map
		\begin{equation*}
			\id_{\lieg} \otimesC \varphi\colon \mfA \rightarrow \lieg \otimesC \End_{\mathbb{C}}(V)
		\end{equation*}
		is an ideal $\ideal_\varphi \subseteq \mfA$. Moreover, $\im(\id_{\lieg} \otimesC \varphi)$ has the structure of a Lie algebra and $\id_{\lieg} \otimesC \varphi$ induces an epimorphism of Lie algebras
		\begin{equation*}
			\mfA \rightarrow \im(\id_{\lieg} \otimesC \varphi)
		\end{equation*}
		with $\im(\id_{\lieg} \otimesC \varphi) \cong \mfA /\ideal_\varphi$.
	\end{prop}
	\begin{proof}
		We will first prove that $\Ker(\id_{\lieg} \otimesC \varphi)=\ideal_\varphi$ is an ideal of $\mfA$. Suppose
		\begin{equation*}
			a=\sum_{A \in \basis} (A \otimes f_A) \in \mfA \text{ and}
		\end{equation*}
		\begin{equation*}
			b=\sum_{B \in \basis} (B \otimes g_B) \in \ideal_\varphi,
		\end{equation*}
		where each $f_A,g_B \in \OX$. It then follows that
		\begin{equation*}
			\sum_{B \in \basis} (B \otimes \varphi(g_B))=0 \Leftrightarrow \varphi(g_B) =0
		\end{equation*}
		for all $B \in \basis$. Thus, we have
		\begin{align*}
			(\id_{\lieg} \otimesC \varphi)([a,b]_{\mfA}) &= \sum_{A \in \basis}\sum_{B \in \basis} ([A,B]_{\lieg} \otimes \varphi(f_A g_B))\\
			&=\sum_{A \in \basis}\sum_{B \in \basis} [A,B]_{\lieg} \otimes \varphi(f_A) \varphi(g_B) \\
			&=0.
		\end{align*}
		So $[\mfA, \ideal_\varphi] \subseteq \ideal_\varphi$, as required. Thus, $\ideal_\varphi$ is an ideal of $\mfA$.
		
		To show that $\im(\id_{\lieg} \otimesC \varphi)$ has the structure of a Lie algebra, for any elements
		\begin{equation*}
			c = \sum_{C \in \basis} (C \otimes p_C) \qquad \text{and} \qquad d=\sum_{D \in \basis} (D \otimes q_D)
		\end{equation*}
		in $\mfA$, we define the Lie bracket on $\im(\id_{\lieg} \otimesC \varphi)$ to be
		\begin{equation*}
			\left[\sum_{C \in \basis} (C \otimes \varphi(p_C)),\sum_{D \in \basis} (D \otimes \varphi(q_D))\right]_{\im(\id_{\lieg} \otimesC \varphi)}=\sum_{C \in \basis}\sum_{D \in \basis} ([C,D]_{\lieg} \otimes \varphi(p_C)\varphi(q_D)).
		\end{equation*}
		Since $\varphi$ is a linear representation of the commutative algebra $\OX$, it is easy to see that $\im(\id_{\lieg} \otimesC \varphi)$ has the structure of a Lie algebra with respect to this Lie bracket and that $\id_{\lieg} \otimesC \varphi$ induces an (epi)morphism of Lie algebras $\mfA \rightarrow \im(\id_{\lieg} \otimesC \varphi)$. So 
		\begin{equation*}
			(\id_{\lieg} \otimesC \varphi)([c,d]_{\mfA})=[(\id_{\lieg} \otimesC \varphi)(c),(\id_{\lieg} \otimesC \varphi)(d)]_{\im(\id_{\lieg} \otimesC \varphi)},
		\end{equation*}
		as required. Finally, we have
		\begin{equation*}
			\im(\id_{\lieg} \otimesC \varphi) \cong \mfA / \Ker(\id_{\lieg} \otimesC \varphi) = \mfA / \ideal_{\varphi}
		\end{equation*}
		by the first isomorphism theorem.
	\end{proof}
	
	Henceforth, we will denote by $\ideal_\varphi$ the ideal of $\mfA$ associated to a representation $\varphi \in \Rep\OX$. That is, $\ideal_\varphi=\Ker(\id_{\lieg} \otimesC \varphi)$ as above. It is fairly straightforward to show that if $\OX$ is finitely generated and reduced with maximal spectrum $\mbX$ (and thus, $\mbX$ has the structure of an affine variety with coordinate ring $\OX$), then for any 1-dimensional representation $\indrep{\genrs_0}{1} \in \Rep\OX$ defined by $\varphi(f)=f(\genrs_0)$ for some $\genrs_0 \in \mbX$, the map $(\id_{\lieg} \otimesC \indrep{\genrs_0}{1})$ coincides with the point-evaluation map $\ev_{\{\genrs_0\}}$. This motivates the following (general) definition.

	\begin{defn}
		Let $\mfA$ be an automorphic Lie algebra and let $\varphi \in \Rep\OX$ be a representation. Let
		\begin{equation*}
			\xi\colon \im(\id_{\lieg} \otimesC \varphi) \rightarrow \mfA / \ideal_\varphi
		\end{equation*}
		be the canonical isomorphism of Lie algebras. We call the linear map
		\begin{equation*}
			\ev_\varphi=\xi \circ (\id_{\lieg} \otimesC \varphi)\colon \mfA \rightarrow \mfA / \ideal_\varphi,
		\end{equation*}
		the \emph{matrix-evaluation map} of $\mfA$ with respect to $\varphi$.
	\end{defn}
	
	We will now investigate some basic properties of these matrix-evaluations. The next lemma shows that the $\Hom$-spaces in $\Rep \OX$ induce relations between the ideals of $\mfA$.
	\begin{lem} \label{MorphRelations}
		Let $\varphi,\psi \in \Rep \OX$ and $\theta \in \Hom_{\OX}(\varphi,\psi)$ be non-zero.
		\begin{enumerate}[label=(\alph*)]
			\item If $\theta$ is a monomorphism then $\ideal_\varphi \supseteq \ideal_\psi$.
			\item If $\theta$ is an epimorphism then $\ideal_\varphi \subseteq \ideal_\psi$.
			\item If $\theta$ is an isomorphism then $\ideal_\varphi=\ideal_\psi$.
		\end{enumerate}
	\end{lem}
	\begin{proof}
		(a) First note that $\theta$ satisfies the relation $\theta \varphi(f) = \psi(f) \theta$ for any $f \in \OX$. Let
		\begin{equation*}
			\sum_{A \in \basis} (A \otimes f_A) \in \Ker(\id_{\lieg} \otimesC \psi) = \ideal_\psi,
		\end{equation*}
		where each $f_A \in \OX$. Then for any $A\in \basis$, we have $\theta \varphi(f_A) = \psi(f_A) \theta=0$. But since $\theta$ is a monomorphism, we have for any $A \in \basis$ that
		\begin{equation*}
			\varphi(f_A)=0 \Rightarrow \sum_{A \in \basis} (A \otimes f_A) \in \Ker(\id_{\lieg} \otimesC \varphi)=\ideal_\varphi
		\end{equation*}
		So $\ideal_\varphi \supseteq \ideal_\psi$, as required.

		(b) The proof is dual to (a).
		
		(c) This follows directly from (a) and (b).
	\end{proof}
	
	Unfortunately, the converse to the above is not true in general. A counter-example is presented below.
	
	\begin{exam} \label{ex:sl2Z5}
		Let $\zeta$ be a primitive fifth root of unity. Consider the automorphic Lie algebra $\mfA=\mathfrak{A}(\sltwo,\mathbb{X}, \rg, \sigma_{\sltwo},\sigma_\mbX)$, where $\mbX=\Rsphere \setminus \{\infty\}$,
		\begin{equation*}
			\rg=\mathbb{Z} / 5 \mathbb{Z}=\langle r \rangle
		\end{equation*}
		and $r$ acts on $\sltwo$ by conjugation with the matrix
		\begin{equation*}
			\begin{pmatrix}
				\zeta 	& 	0 \\
				0		&	\zeta^{-1}
			\end{pmatrix}
		\end{equation*}
		and on $\Rsphere$ by multiplication with $\zeta$. Then $\OX \cong \mathbb{C}[\polyvar]$ and $\mfA$ has a basis
		\begin{equation*} 
			\left\{
			\begin{pmatrix}
				\polyvar^{5j} 	& 	0 \\
				0		&	-\polyvar^{5j}
			\end{pmatrix},
			\begin{pmatrix}
				0 	& 	\polyvar^{5j+3} \\
				0	&	0
			\end{pmatrix},
			\begin{pmatrix}
				0 			& 	0 \\
				\polyvar^{5j+2}	&	0
			\end{pmatrix}
			: j \geq 0
			\right\}.
		\end{equation*}
		Now consider the representations $\indrep{0}{1} \subset \indrep{0}{2} \subset \indrep{0}{3} \in \fin \OX$ defined in Section~\ref{MCGammaReps}. Since $\indrep{0}{m}^m=0$, it is easy to see that
		\begin{align*}
			\ideal_{\indrep{0}{1}}&=\left\langle
			\begin{pmatrix}
				\polyvar^{5(j+1)} 	& 	0 \\
				0		&	-\polyvar^{5(j+1)}
			\end{pmatrix},
			\begin{pmatrix}
				0 	& 	\polyvar^{5j+3} \\
				0	&	0
			\end{pmatrix},
			\begin{pmatrix}
				0 			& 	0 \\
				\polyvar^{5j+2}	&	0
			\end{pmatrix}
			: j \geq 0
			\right\rangle, \\
			\ideal_{\indrep{0}{2}}&=\ideal_{\indrep{0}{1}} \\
			\ideal_{\indrep{0}{3}}&=\left\langle
			\begin{pmatrix}
				\polyvar^{5(j+1)} 	& 	0 \\
				0		&	-\polyvar^{5(j+1)}
			\end{pmatrix},
			\begin{pmatrix}
				0 	& 	\polyvar^{5j+3} \\
				0	&	0
			\end{pmatrix},
			\begin{pmatrix}
				0 			& 	0 \\
				\polyvar^{5(j+1)+2}	&	0
			\end{pmatrix}
			: j \geq 0
			\right\rangle,
		\end{align*}
		and yet $\indrep{0}{1} \not\cong \indrep{0}{2}$.
	\end{exam}
	
	Any finite-dimensional representation of $\OX$ can be written as a finite direct sum of indecomposable representations. Thus, it is natural to investigate the relations between the ideals given by a direct sum of representations of $\OX$.
	\begin{lem} \label{DirectSumIdeal}
		Let $\varphi \in \Rep \OX$ and suppose $\varphi = \phi \oplus \psi$ for some representations $\phi,\psi \in \Rep\OX$. Then $\ideal_\varphi=\ideal_\phi \cap \ideal_\psi$.
	\end{lem}
	\begin{proof}
		Suppose
		\begin{equation*}
			a=\sum_{A \in \basis} (A \otimes f_A) \in \Ker(\id_{\lieg} \otimesC \varphi) = \ideal_\varphi,
		\end{equation*}
		where each $f_A \in \OX$. Recall that 
		\begin{equation*}
			\sum_{A \in \basis} (A \otimes \varphi(f_A)) = 0 \Leftrightarrow \varphi(f_A)=0
		\end{equation*}
		for any $A \in \basis$. But since $\varphi= \phi \oplus \psi$, this means that $(\phi(f_A),\psi(f_A))=0$, which is true if and only if $\phi(f_A)=0$ and $\psi(f_A)=0$. Thus, $a \in \ideal_\phi$ and $a \in \ideal_\psi$, and hence $a \in \ideal_\phi \cap \ideal_\psi$. The converse argument is similar, and so $\ideal_\varphi=\ideal_\phi \cap \ideal_\psi$.
	\end{proof}
	
	\begin{rem}
		A consequence of Lemma~\ref{DirectSumIdeal} is that 
		\begin{equation*}
			[\ideal_\varphi, \ideal_\psi] \subseteq \ideal_{\varphi \oplus \psi} \subseteq \ideal_\varphi + \ideal_\psi.
		\end{equation*}
	\end{rem}

	\subsection{Ideals Arising from Jet Representations} \label{sec:ChainsOfIdeals}
	In this subsection, we investigate the ideals and quotients of $\mfA$ that arise from the jet representations $\jetrep{\genrs}{m}$ of $\OX$, as defined in Section~\ref{sec:RepsFromJets}. For the purposes of readability, we will denote for some $\genrs_0 \in \mbX$ and $m \in \strictpos$ the ideal $\ideal_{\genrs_0,m}$ of $\mfA$ associated to the representation $\jetrep{\genrs_0}{m} \in \Rep\OX$ (as defined in Section~\ref{sec:RepsFromJets}).
	
	\begin{prop}\label{IdealInvariant}
		Given an automorphic Lie algebra $\mfA=\generalALA$, we have $\ideal_{\gamma \genrs_0,m}=\ideal_{\genrs_0,m}$ for any $\gamma \in \rg$.
	\end{prop}
	\begin{proof}
		Let $a=\sum_{A \in \basis} A \otimes f_A \in \ideal_{\genrs_0,m}$. We begin the proof with the claim that $(\gamma f_A)^{(k)}(\gamma \genrs_0)=0$ for any $0 \leq k < m$.
		
		For $k=0$, the proof of this claim is straightforward. We simply note that 
		\begin{equation*}
			(\gamma f_A)(\gamma \genrs_0)=f_A(\gamma^{-1}\gamma \genrs_0)=f_A(\genrs_0)=0
		\end{equation*}
		for any $A \in \basis$. Now for any $k \geq 1$, it follows that
		\begin{equation*}
			(\gamma f_A)^{(k)} = (f_A \circ \gamma^{-1})^{(k)}=\sum_{i=1}^k (f_A^{(i)} \circ \gamma^{-1})p_i,
		\end{equation*}
		where each $p_i$ is some holomorphic function. But by Remark~\ref{JetMatFunction}, we know that $f_A^{(i)}(\genrs_0)=0$ for all $0 \leq k < m$. So
		\begin{equation*}
			(\gamma f_A)^{(k)}(\gamma \genrs_0) = \sum_{i=1}^k f_A^{(i)}(\genrs_0)p_i(\gamma\genrs_0)=0
		\end{equation*}
		for any $1 \leq k \leq m-1$, which completes the proof for our claim above.
		
		A consequence of the above claim is that $a \in \ideal_{\gamma \genrs_0, m}$, since $\mfA$ is a Lie algebra of fixed points and $\gamma$ acts on $\lieg$ by automorphisms. So $\ideal_{\genrs_0, m} \subseteq \ideal_{\gamma \genrs_0, m}$ for any $\gamma \in \rg$, any $\genrs_0 \in \mbX$ and any $m \in \strictpos$. To show that $\ideal_{\genrs_0, m} \supseteq \ideal_{\gamma \genrs_0, m}$ and hence that $\ideal_{\genrs_0, m} = \ideal_{\gamma \genrs_0, m}$, we simply note that by an identical argument we have $\ideal_{\gamma \genrs_0, m} \subseteq \ideal_{\gamma^{-1} \gamma \genrs_0, m} = \ideal_{\genrs_0, m}$.
	\end{proof}
	
	One can construct various chains of ideals of $\mfA$, and obtain information regarding their structure. As a result of Section~\ref{sec:RepsFromJets}, an immediate consequence of the Lemma~\ref{MorphRelations} is the following.
	
	\begin{lem} \label{IdealChain}
		Let $\jetrep{\genrs_0}{1} \in \fin \OX$ be an irreducible representation as defined in Section~\ref{sec:RepsFromJets}. Then there exists a weakly descending chain of ideals
		\begin{equation*}
			\mfA \supset \ideal_{\genrs_0,1} \supseteq \ideal_{\genrs_0,2} \supseteq \ideal_{\genrs_0,3} \supseteq \ldots
		\end{equation*}
		of $\mfA$.
	\end{lem}
	\begin{proof}
		This follows trivially from the composition series of finite-dimensional representations of $\OX$ given in Section~\ref{sec:RepsFromJets} (with $m$ sufficiently large), since this induces a sequence of monomorphisms
		\begin{equation*}
			\jetrep{\genrs_0}{1} \rightarrow \jetrep{\genrs_0}{2} \rightarrow \jetrep{\genrs_0}{3} \rightarrow \cdots
		\end{equation*}
		in $\fin \OX$.
	\end{proof}

	\begin{lem} \label{lem:IdealsNonTrivial}
		Let $\mfA=\generalALA$ be an automorphic Lie algebra. Then for any $\genrs_0 \in \rs$ and any $m \in \strictpos$, the ideal $\ideal_{\genrs_0,m} \neq 0,\mfA$.
	\end{lem}
	\begin{proof}
		First note that $\ideal_{\genrs_0,1} \neq 0$ and $\ideal_{\genrs_0,1} \neq \mfA$, since $\mfA / \ideal_{\genrs_0,1} \cong \lieg^{\genrs_0}$, which is a non-trivial quotient of $\mfA$ by the results of \cite{neher2012irreducible}. Now Lemma~\ref{IdealChain} implies that $\ideal_{\genrs_0,m} \subseteq \ideal_{\genrs_0,1}$ for any $m>1$. So trivially, we have $\ideal_{\genrs_0,m} \neq \mfA$. It remains to show that $\ideal_{\genrs_0,m}\neq 0$, or equivalently, that the quotient $\mfA / \ideal_{\genrs_0,m} \neq \mfA$ . By Proposition~\ref{prop:IdealsFromReps}, the quotient $\mfA / \ideal_{\genrs_0,m} \cong \im (\id_{\lieg} \otimesC \jetrep{\genrs_0}{m})$, which is necessarily finite-dimensional, since both $\lieg$ and $\jetrep{\genrs_0}{m}$ are finite-dimensional. But $\mfA$ is infinite-dimensional. Thus, $\mfA / \ideal_{\genrs_0,m} \neq \mfA$ and hence, $\ideal_{\genrs_0,m} \neq 0$ as required.
	\end{proof}

	\begin{lem} \label{ProperBase}
		Let $\mfA=\generalALA$ be an automorphic Lie algebra. Then for any $\genrs_0 \in \mbX$ and any $n \in \strictpos$, there exists an integer $N>n$ such that $\ideal_{\genrs_0,N} \neq \ideal_{\genrs_0,n}$.
	\end{lem}
	\begin{proof}
		Let $n \in \strictpos$ and consider a non-zero element
		\begin{equation*}
			a =\sum_{A \in \basis} A \otimes f_A \in \ideal_{\genrs_0,n}
		\end{equation*}
		whose existence we may assume due to Lemma~\ref{lem:IdealsNonTrivial}. Trivially, we know that $a \in \ideal_{\genrs_0,n}$ if and only if $\jetrep{\genrs_0}{n}(f_A) = 0$ for all $A \in \basis$. In fact, it follows from the definition of $\jetrep{\genrs_0}{n}$ and Remark~\ref{JetMatFunction} that $a \in \ideal_{\genrs_0,n}$ if and only if $f_A^{(r)}(\genrs_0)= 0$ for all $A \in \basis$ and all $0 \leq r < n$.
		
		From the fact that $f_A(\genrs_0)= 0$ for all $A \in \basis$, it necessarily follows that there exists $B \in \basis$ such that $f_B$ is non-constant. To see this, note that if $f_A$ were constant for every $A \in \basis$, then there would necessarily exist $B \in \basis$ such that $f_B \not\equiv 0$ (otherwise we would contradict the assumption that $a \neq 0$). But if $f_B$ is constant non-zero, then $f_B(\genrs_0) \neq 0$, which would contradict the assumption that $a \in \ideal_{\genrs_0,n}$.
		
		Finally, since each $f_A$ is holomorphic on $\rs$, each $f_A$ is analytic at $\genrs_0$. But $f_B$ is a non-constant function with $f_B^{(r)}(\genrs_0) = 0$ for all $0 \leq r < n$. Thus, there must exist an integer $N > n$ such that $f_B^{(N-1)}(\genrs_0) \neq 0$. Hence, $a \not\in \ideal_{\genrs_0,N}$ and so $\ideal_{\genrs_0,N} \neq \ideal_{\genrs_0,n}$, as required.
	\end{proof}
	
	\begin{prop} \label{ProperChain}
		Let $\mfA=\generalALA$ be an automorphic Lie algebra. Then for any $\genrs_0 \in \rs$, there exist strictly positive integers $m_1< m_2< m_3<\ldots$ such that we have a strictly descending chain of ideals
		\begin{equation*}
			\mfA \supset \ideal_{\genrs_0,m_1} \supset \ideal_{\genrs_0,m_2} \supset \ideal_{\genrs_0,m_3} \supset \ldots
		\end{equation*}
		of $\mfA$.
	\end{prop}
	\begin{proof}
		This follows immediately from Lemmata~\ref{IdealChain}, \ref{lem:IdealsNonTrivial} and \ref{ProperBase}.
	\end{proof}
	
		\section{The Wildness of Automorphic Lie Algebras}
	\label{sec:wildness}
	In this section, we show that automorphic Lie algebras are wild, and hence, the problem of classifying the indecomposable representations may be considered hopeless. The wildness immediately follows from the representation theory of $\OX$. Interestingly, the wildness of an automorphic Lie algebra $\mfA$ can arise from the ideals and quotients of $\mfA$ that exist as a result of certain representations of $\OX$ given by nilpotent matrices. For example, these exist whenever we have representations of $\OX$ given by Jordan blocks, as is the case for many commutative algebras (including those that are tame).
The results of this section also hold for twisted/untwisted loop algebras and should be easily adaptable to equivariant map algebras.
	
	\begin{prop} \label{SolvableIdeals}
		Let $\mfA = \generalALA$ be an automorphic Lie algebra, let $\genrs_0 \in \mbX$ and let $\jetrep{\genrs_0}{1} \in \fin \OX$ be as defined in Section~\ref{sec:RepsFromJets}. Let $m$ be the least integer such that $\ideal_{\genrs_0,m} \neq \ideal_{\genrs_0,1}$. Then for any $n \geq m$, the quotient $\mfA / \ideal_{\genrs_0,n}$ is a finite-dimensional Lie algebra with a non-zero solvable ideal. Moreover, for any $N>0$, there exists an integer $n \geq m$ such that $\mfA / \ideal_{\genrs_0,n}$ has a solvable ideal $\mathfrak{S}$ with $\dim \mathfrak{S}>N$.
	\end{prop}
	\begin{proof}
		By Lemma~\ref{IdealChain} and the third isomorphism theorem, it follows that $\mathfrak{K}=\ideal_{\genrs_0,1} /\ideal_{\genrs_0,m}$ is a non-trivial ideal of $\mfA / \ideal_{\genrs_0,m}$. Let $\basis$ be a basis for $\lieg$ and let
		\begin{equation*}
			\sum_{A \in \basis} A \otimes f_A \in \mathfrak{K}.
		\end{equation*}
		The unique eigenvalue of $\jetrep{\genrs_0}{1}(f_A)$ is trivially $0$ for any $A$, so by Remark~\ref{JetMatFunction}, the unique eigenvalue of the matrix $\jetrep{\genrs_0}{m}(f_A)$ is also $0$ for any $A$. Moreover, $\jetrep{\genrs_0}{m}(f_A)$ is a banded upper triangular matrix. So $\jetrep{\genrs_0}{m}(f_A)$ is nilpotent for any $A \in \basis$. Since $m$ is minimal, the only non-zero entry of $\jetrep{\genrs_0}{m}(f_A)$ is the $(1,m)$-th entry, and we have $\jetrep{\genrs_0}{m}(f_A)\jetrep{\genrs_0}{m}(g_A)=0$ for all $A$ for any
		\begin{equation*}
			a=\sum_{A \in \basis} A \otimes f_A, b=\sum_{B \in \basis} B \otimes g_B \in \mathfrak{K}.
		\end{equation*}
		Thus, $(\id_{\lieg} \otimesC \jetrep{\genrs_0}{m})([a,b])=0$. So $\mathfrak{K}$ is abelian and hence solvable.
		
		By Proposition~\ref{ProperChain}, we may choose another integer $m'>m$ such that $\ideal_{\genrs_0,m'} \neq \ideal_{\genrs_0,m}$. Let $m'$ be minimal. Then a similar argument to the above, shows that $\mathfrak{K}'=\ideal_{\genrs_0,m} /\ideal_{\genrs_0,m'}$ is a non-trivial abelian ideal of $\mfA / \ideal_{\genrs_0,m'}$. In fact, it is easy to show that $\ideal_{\genrs_0,1} /\ideal_{\genrs_0,m'}$ is solvable, since every matrix induced by the representation $\jetrep{\genrs_0}{m'}$ is nilpotent. Proceeding iteratively, it follows that one can construct a quotient of $\mfA$ that contains a solvable ideal of an arbitrarily large dimension. Thus, for any $N>0$, there exists an integer $n \geq m$ such that $\mfA / \ideal_{\genrs_0,n}$ has a solvable ideal $\mathfrak{S}$ with $\dim \mathfrak{S}>N$.
	\end{proof}
	
	\begin{thm}
\label{thm:wildness}
		Automorphic Lie algebras are wild.
	\end{thm}
	\begin{proof}
		Let $\mfA$ be an automorphic Lie algebra. By Proposition~\ref{SolvableIdeals}, for any $N>1$, there exists a finite-dimensional quotient $\mathfrak{Q}$ of $\mfA$ such that $\mathfrak{Q}$ contains a solvable ideal $\mathfrak{S}$ with $\dim \mathfrak{S}>N$. Thus, $\mathfrak{Q}$ is neither semisimple nor a 1-dimensional extension of a semisimple Lie algebra. By Makedonski\u{\i}'s Theorem (Theorem~\ref{MakedonskiiThm}), $\mathfrak{Q}$ is wild. Thus, $\mfA$ is wild.
	\end{proof}
	
\newcommand{\mev}[2]{\mathrm{ev}_{\phi_{#1,#2}}}
\newcommand{\omev}[2]{\id\otimes_{\mb{C}}\phi_{#1,#2}}

\newcommand {\rd}{\mathrm d}

\newcommand{\cp}{\overline{\mb{C}}}
\newcommand{\zn}[1]{\mb{Z}/#1\mb{Z}}
\newcommand{\ord}[2]{\text{ord}_{#1}\!\left(#2\right)}

\newcommand{\roots}{\Phi}
\newcommand{\sroots}{\Delta}

\newcommand{\SL}{\mathrm{SL}}
\newcommand{\Sp}{\mathrm{Sp}}
\newcommand{\PSL}{\mathrm{PSL}}
\newcommand{\Int}[1]{\mathrm{Int}\!\left(#1 \right)}
\newcommand{\Ad}{\mathrm{Ad}}
\newcommand{\res}{\mathrm{res}\,}

\newcommand{\dynkinscale}{0.7}
\newcommand{\dynkintablescale}{0.5}
\newcommand{\dynkinfont}{\footnotesize}
\tikzstyle{dynkinnode}=[draw, color=blue, shape=circle,minimum size=3.5 pt,inner sep=0]
\tikzstyle{ldynkinnode}=[draw, color=blue, shape=circle,minimum size=3.5 pt,inner sep=0]
\tikzstyle{sdynkinnode}=[draw, color=red, shape=circle,minimum size=3.5 pt,inner sep=0]
\tikzstyle{mdynkinnode}=[draw, color=magenta, shape=circle,minimum size=3.5 pt,inner sep=0]
\tikzstyle{fdynkinnode}=[draw, color=blue, shape=circle,minimum size=3.5 pt,inner sep=0,fill=black]
\tikzstyle{fldynkinnode}=[draw, color=blue, shape=circle,minimum size=3.5 pt,inner sep=0,fill=black]
\tikzstyle{fsdynkinnode}=[draw, color=red, shape=circle,minimum size=3.5 pt,inner sep=0,fill=black]
\tikzstyle{fmdynkinnode}=[draw, color=magenta, shape=circle,minimum size=3.5 pt,inner sep=0,fill=black]
\tikzstyle{brace}=[ decorate, decoration={brace, amplitude=5pt}]
\tikzstyle{mbrace}=[decorate, decoration={brace, amplitude=5pt, mirror}]

\newcommand{\finitegroupfont}{\mathsf}
\newcommand{\cg}[1]{\finitegroupfont{C}_{#1}}
\newcommand{\dg}[1]{\finitegroupfont{D}_{#1}}
\newcommand{\tg}{\finitegroupfont{T}}
\newcommand{\og}{\finitegroupfont{O}}
\newcommand{\yg}{\finitegroupfont{Y}}

\section{Twisted Truncated Current Algebras as Quotients of Automorphic Lie Algebras}
\label{sec:closer} 
We continue to investigate the quotients $\mf{A}/\mf{I}_{\genrs_0,m}$ of an automorphic Lie algebra defined by the jet representations of Section \ref{sec:RepsFromJets}.
A careful choice of the local coordinate $z$ near $x_0$ together with the Riemann-Roch theorem allows us to prove that these quotients are isomorphic to twisted truncated current algebras.

The careful choice of $z$ is the one that linearises the action of $\rg$ locally in the following sense. Throughout, we fix a point $\genrs_0\in\rs$, which defines a stabiliser subgroup $\rg_{\genrs_0}$ generated by an element $\gamma_0\in\rg$ of order $\nu_0\in\mb{N}$. We then fix a coordinate $z$ on $\rs$ in a neighbourhood of $\genrs_0$ and a root of unity $\zeta$ such that
\begin{align}
\label{eq:z}&z(\genrs_0)=0,\quad \gamma_0\cdot z(\genrs)= z(\gamma_0^{-1}\genrs)=\zeta^{-1} z(\genrs)
\end{align} 
for all $\genrs$ in the domain of $z$.
A proof of the fact that $\rg_{\genrs_0}$ is cyclic (and finite), and of the existence
of the coordinate $z$, can be found in \cite{miranda1995algebraic}\footnote{In the genus zero case we can construct the coordinate as follows. An automorphism $\gamma_0=[g]\in\PSL(2,\mb{C})$ fixes a point $[\genrs_0:1]$ if and only if $(\genrs_0,1)^t$ is a right eigenvector of $g\in\SL(2,\mb{C})$. This is the case if and only if $(1,-\genrs_0)$ is a left eigenvector of $g^{-1}$ (using that $\SL(2,\mb{C})=\Sp(2,\mb{C})$ with respect to the standard symplectic form). Because $g$ has finite order, there is a left eigenvector $(a,b)$ of $g^{-1}$ independent of $(1,-\genrs_0)$. The coordinate $z(\genrs)=\frac{\genrs-\genrs_0}{a\genrs+b}$ satisfies the two conditions.}.

The Lie structure will not be used until the end of the section, so we will replace $\mf{g}$ with a complex finite dimensional vector space $V$, and write the chain of subspaces
\begin{equation}
\label{eq:chain of subspaces}
(V\otimes_{\mb{C}}\mc{O}_\rs)^\rg=I_{\genrs_0,0}\supset I_{\genrs_0,1}\supset I_{\genrs_0,2}\supset I_{\genrs_0,3}\ldots
\end{equation}
where $I_{\genrs_0,m}$ consists precisely of those vectors in $(V\otimes_{\mb{C}}\mc{O}_\rs)^\rg$ with derivatives of order $0,1,\ldots,m-1$ vanishing at $\genrs_0$. 

We denote the eigenspace decomposition of $V$ with respect to $\Gamma_{\genrs_0}$ by 
\[V=\bigoplus_{\bar{m}\in\zn{\nu_0}}V_{\bar{m}},\quad V_{\bar{m}}=\{v\in V\,|\,\gamma_0 v =\zeta^m v\}\]
where $V$ is a finite dimensional $\rg$-module.

For the study of the quotients $I_{\genrs_0,m}/I_{\genrs_0,m+1}$ it is useful to notice that $I_{\genrs_0,m+1}$ is the kernel of the restriction of $\vjet{\genrs_0}{m}$
 to $I_{\genrs_0,m}$. Thus the first isomorphism theorem gives
\begin{equation}
\label{eq:inclusions}
I_{\genrs_0,m}/I_{\genrs_0,m+1}\cong \vjet{\genrs_0}{m}(I_{\genrs_0,m})\subset V_{\bar{m}}\otimes z^m+(z^{m+1})
\end{equation}
and the last inclusion follows from invariance under $\rg_{\genrs_0}$ and the transformation (\ref{eq:z}) of $z$. We will see that the inclusion is in fact an equality. To this end we will use the following application of the Riemann-Roch theorem. 
\begin{lem}[holomorphic interpolation]\label{lem:interpolation}
Let $Z=\{z_1,\ldots,z_d\}$ be a subset of $\rs$ and $\{c_1,\ldots,c_d\}$ a subset of $\mb{C}$. For any integer $m\ge 0$ there exists $f\in \mc{O}_\rs$ such that $f^{(m)}(z_i)=c_i$ and $f^{(j)}(z_i)=0$ if $j<m$.
\end{lem}
\begin{proof}
\renewcommand{\div}[1]{\text{div}\!\left(#1\right)}
\newcommand{\mero}[1][\cp]{\mathcal{M}(#1)}

For a divisor $D$ on a Riemann surface $\srs$ one defines a vector space $L(D)$ of meromorphic functions on $\srs$ by
\[L(D)=\{f\text{ meromorphic on }\srs\;|\; \div{f}+D\ge0\}.\]
If a point $p\in\srs$ has a positive coefficient $n$ in $D$ then $f\in L(D)$ is allowed a pole at $p$ of order at most $n$. If a point $p\in\srs$ has a negative coefficient $-n$ in $D$ then $f\in L(D)$ has a zero at $p$ of order at least $n$. If $f$ is any meromorphic function on a compact Riemann surface then $\deg(\div{f})=0$. Consequently $L(D)=\{0\}$ if $\deg(D)<0$ and $\srs$ compact.

The Riemann-Roch theorem for compact Riemann surfaces states
\begin{equation}\label{eq:RR}
\dim L(D)-\dim L(K-D)=\deg(D)+1-g.
\end{equation}
Here $K$ is a canonical divisor: a divisor of a $1$-form on $\srs$, and $g$ is the genus of the Riemann surface. 

With the notation $D_S=\sum_{s\in S}s$ we define the divisors
\begin{align*}
D^m_n&=nD_\poles-m D_Z,
\\\tilde{D}^m_n&=nD_\poles-m D_{Z\setminus\{z_1\}}-(m-1)z_1.
\end{align*}
Notice that  \[L({D}^m_n)\subset L(\tilde{D}^m_n)\subset\mc{O}_\rs\] for all $m\in\mb{Z}_{>0}$ and $n\in\mb{Z}$. 

Because the order of poles is irrelevant for our current purposes, we can choose $n\in \mb{N}$ large enough so that $\deg(K-D^m_n)<0$ and consequently $\dim L(K-D^m_n)=0$. Doing the same for $\tilde{D}^m_n$ reduces the Riemann-Roch equation (\ref{eq:RR}) to 
\begin{align*}
\dim L({D}^m_n)&=n|\poles|-m|Z|+1-g
\\\dim L(\tilde{D}^m_n)&=n|\poles|-m(|Z|-1)-(m-1)+1-g
\end{align*}
and $\dim L(\tilde{D}^m_n)-\dim L({D}^m_n)=1$. This shows that there exists a function $f_1\in \mc{O}_\rs$ such that $f_1^{(j)}(z_i)=0$ if $j<m$ and $f_1^{(m)}(z_1)=c_1$ and $f_1^{(m)}(z_i)=0$ for $i\ne 1$.

The argument can be repeated with $z_1$ replaced by $z_k$ to obtain $f_k$ with the analogous properties. Then $f=\sum_k f_k$ satisfies the requirements of the lemma.
\end{proof}

Using holomorphic interpolation we can describe a basis for each of the factors of (\ref{eq:chain of subspaces}). We let $O(z^i)$ denote terms divisible by $z^i$.
\begin{thm}
\label{thm:composition factors}
Let $m\ge0$ and let $\{v^{\bar{m}}_1,\ldots,v^{\bar{m}}_{k_{\bar{m}}}\}$ 
be a basis of $V_{\bar{m}}$. 
Then there exists a set $\{a^{\bar{m}}_1,\ldots,a^{\bar{m}}_{k_{\bar{m}}}\}$ contained in $I_{\genrs_0,m}$ with Taylor expansion \[a^{\bar{m}}_i=v^{\bar{m}}_i\otimes z^m+O(z^{m+1})\]
such that  $\{a^{\bar{m}}_1+I_{\genrs_0,m+1},\ldots,a^{\bar{m}}_{k_{\bar{m}}}+I_{\genrs_0,m+1}\}$ is a basis of $I_{\genrs_0,m}/I_{\genrs_0,m+1}$.
\end{thm}

\begin{proof}
Suppose the statement holds for the regular representation $V=\mb{C}\rg$. Since $\rg$ is finite, $\mb{C}\rg$ decomposes into a direct sum containing any irreducible of $\rg$ as a summand.
Let $V_i$ be an irreducible $\rg$-module. Then the equivariant vectors with values in $V_i$ form a summand of the equivariant vectors with values in $\mb{C}\rg$, and therefore the statement also holds for each irreducible. But then it holds for any $\rg$-module $V$.

We will prove the statements for the regular representation $V=\mb{C}\rg$. 
Fix a vector \[v=\sum_{\gamma\in \rg} c_\gamma \gamma\in (\mb{C}\rg)_{\bar{m}}.\] This means that $\gamma_0\sum_{\gamma\in \rg} c_\gamma \gamma=\sum_{\gamma\in \rg} c_\gamma \gamma_0\gamma=\zeta^{m}\sum_{\gamma\in \rg} c_\gamma \gamma$ or that 
\begin{equation}
\label{eq:csubgamma}
c_{\gamma_0^{-1}\gamma}=\zeta^{m} c_\gamma.
\end{equation}
We will show existence of an element $F$ of $I_{\genrs_0,m}$ with $v$ as first coefficient in the Taylor expansion. That is 
\begin{equation}
\label{eq:F}
F=v\otimes_{\mb{C}}(\genrs-\genrs_0)^m+O((\genrs-\genrs_0)^{m+1})\in I_{\genrs_0,m}.
\end{equation}

From any function $f\in\mc{O}_\rs$ one can construct an invariant vector 
\[F =\sum_{\gamma\in\rg} \gamma\otimes f\circ \gamma^{-1}\in (\mb{C}\rg \otimes_\mb{C}\mc{O}_\rs)^\rg.\] To investigate its Taylor expansion about $\genrs_0$ we compute derivatives $(f\circ \gamma^{-1})^{(j)}(\genrs_0)$ for $j\le m$. It is now convenient to write the action of $\rg$ on $\rs$ in terms of a representation $\sigma:\rg\rightarrow\Aut(\rs)$. 
\begin{align*}
(f\circ \gamma^{-1})^{(j)}(\genrs_0)=&
(f^{(j)}\circ \gamma^{-1})(\genrs_0)(\sigma(\gamma^{-1})'(\genrs_0))^j
\\&+
\text{terms with factor }(f^{(j')}\circ \gamma^{-1})(\genrs_0)\text{ with }j'<j.
\end{align*}
This expression leads us to define a new set of constants 
\begin{equation}
\label{eq:csupgamma}
c^\gamma=m!\,(\sigma(\gamma^{-1})'(\genrs_0))^{-m} \,c_\gamma.
\end{equation}
We aim to interpolate the data $(\gamma^{-1}\genrs_0,c^\gamma)$ using Lemma \ref{lem:interpolation} to show existence of $f\in\mc{O}_\rs$ with 
\begin{equation}
\label{eq:interpolatingf}
f^{(m)}(\gamma^{-1}\genrs_0)=c^\gamma,\quad f^{(j)}(\gamma^{-1}\genrs_0)=0,\quad j<m.
\end{equation}
To do so we need to confirm that if $\gamma^{-1}\genrs_0=\tilde{\gamma}^{-1}\genrs_0$ then $c^\gamma=c^{\tilde{\gamma}}$. But $\gamma^{-1}\genrs_0=\tilde{\gamma}^{-1}\genrs_0$ if and only if $\tilde{\gamma}\gamma^{-1}\in\rg_{\genrs_0}=\langle\gamma_0\rangle$ so it is sufficient to show that $\gamma\rightarrow c^\gamma$ is constant on the classes $\rg/\rg_{\genrs_0}$. We compute
\begin{align*}
c^{\gamma_0^{-1}\gamma}
&=m!\,(\sigma((\gamma_0^{-1}\gamma)^{-1})'(\genrs_0))^{-m} \,c_{\gamma_0^{-1}\gamma}
\\&=m!\,((\sigma(\gamma^{-1})\sigma(\gamma_0))'(\genrs_0))^{-m} \,c_{\gamma_0^{-1}\gamma}
\\&=m!\,(\sigma(\gamma^{-1})'(\sigma(\gamma_0)(\genrs_0))\sigma(\gamma_0)'(\genrs_0))^{-m} \,c_{\gamma_0^{-1}\gamma}
\\&=m!\,(\sigma(\gamma^{-1})'(\genrs_0)\zeta)^{-m} \,\zeta^{m} c_{\gamma}=c^\gamma
\end{align*}
using (\ref{eq:z}) and (\ref{eq:csubgamma}) in the last line.
We conclude that $f$ exists. Now $F=\sum_{\gamma\in\rg} \gamma\otimes f\circ \gamma^{-1}$ satisfies (\ref{eq:F}).

We have shown that $I_{\genrs_0,m}/I_{\genrs_0,m+1}$ contains a subset $B=\{a^{\bar{m}}_1+I_{\genrs_0,m+1},\ldots,a^{\bar{m}}_{k_{\bar{m}}}+I_{\genrs_0,m+1}\}$ of elements whose set of first Taylor coefficients corresponds to any choice of basis of $V_{\bar{m}}$. 
Elements of $B$ are linearly independent, since their first Taylor coefficients are. Moreover, from (\ref{eq:inclusions}) we know that the dimension of $I_{\genrs_0,m}/I_{\genrs_0,m+1}$ is at most $k_{\bar{m}}$. Hence $B$ is a basis.
\end{proof}

Theorem \ref{thm:composition factors} and Gaussian elimination results in the following reformulation.
\begin{cor}
$(V\otimes_{\mb{C}}\mc{O}_{\rs})^\rg/{I}_{\genrs_0,m}\cong \left(V\otimes\mb{C}[z]/(z^m)\right)^{\gamma_0}\cong \left(V\otimes\mb{C}[z]\right)^{\gamma_0}/(z^m).$
\end{cor}
We reintroduce a complex finite dimensional Lie algebra $\mf{g}$ in place of $V$ and use Proposition \ref{prop:IdealsFromReps} to see that the linear isomorphism becomes an isomorphism of Lie algebras. Thus we get
\begin{cor}
\label{cor:twisted truncated current algebra} The quotient of an automorphic Lie algebra $\mf{A}$ on a punctured compact Riemann surface $\rs$ by the ideal $\mf{I}_{\genrs_0,m}$ has the form
\[\mf{A}/\mf{I}_{\genrs_0,m}\cong \left(\mf{g}\otimes\mb{C}[z]/(z^m)\right)^{\gamma_0}\cong \left(\mf{g}\otimes\mb{C}[z]\right)^{\gamma_0}/(z^m).\]
\end{cor}
In particular, if $A_1^{\bar{k}},\ldots,A_{d_{\bar{k}}}^{\bar{k}}$ is a basis of $\mf{g}_{\bar{k}}$ for $\bar{k}\in\zn{\nu_0}$, then a basis for the twisted truncated current algebra is given by 
\[A_j^{\bar{k}}\otimes z^k + (z^m),\quad k=0,\ldots,m-1, \quad j=1,\ldots,d_{\bar{k}}.\]
The power of $z$ defines a $\mb{Z}$-grading \[\mf{A}/\mf{I}_{\genrs_0,m}=\bigoplus_{i=0}^{\infty} \mf{A}_i,\quad [\mf{A}_i,\mf{A}_j]\subset\mf{A}_{i+j},\quad \mf{A}_i=\{0\} \text{ if }i\ge m.\]

The case $m=\nu_0$ is special because the dimension of $\mf{A}/\mf{I}_{\genrs_0,\nu_0}$ equals that of $\mf{g}$. 
We can in fact see that $\mf{A}/\mf{I}_{\genrs_0,\nu_0}$ is a contraction\footnote{For $z\in(0,1]$, the map $U_z:A_j^{\bar{k}}\mapsto A_j^{\bar{k}} z^{k}$, $k=0,\ldots,\nu_0-1$, $j=1,\ldots,d_{\bar{k}}$ is a linear isomorphism of $\mf{g}$ which sends the Lie structure to an equivalent one. The limit $[A,B]_0=\lim_{z\mapsto 0} U_z[U_z^{-1} A,U_z^{-1} B]$ exists for all $A,B\in\mf{g}$ and defines a new Lie structure on $\mf{g}$. The quotient $\mf{A}/\mf{I}_{\genrs_0,\nu_0}$ is isomorphic to $(\mf{g},[\cdot,\cdot]_0)$.} of $\mf{g}$.

For all but finitely many $x_0\in\rs$ the automorphism $\gamma_0$ is trivial. The corresponding quotients $\mf{g}\otimes\mb{C}[z]/(z^m)$ of the automorphic Lie algebras are known as twisted truncated current algebras. Their representation theory has seen great developments in recent years \cite{wilson2007representations,wilson2011highest,fourier2015new,kus2015fusion,bianchi2018bases}.

We end this section with a simple example of an automorphic Lie algebra of genus $1$ illustrating Theorem \ref{thm:composition factors} and its Corollary \ref{cor:twisted truncated current algebra}.
\begin{exam}
\newcommand{\lat}{L}
Let $\lat=\mb{Z}l_1\oplus\mb{Z}l_2$ be a lattice in $\mb{C}$ and $\mb{T}=\mb{C}/\lat$ the associated torus with complex structure. It is well known that the field of meromorphic functions on $\mb{T}$ is $\mb{C}(\wp)\oplus\mb{C}(\wp)\wp'$ where $\wp$ is the Weierstrass p function
\[\wp(z)=\frac{1}{z^2}+\sum_{0\ne l\in\lat}\frac{1}{(z+l)^2}-\frac{1}{l^2}\]
which satisfies the equation $(\wp')^2=4\wp^3-g_2\wp-g_3$ where $g_2=60 \sum_{0\ne l\in\lat}\frac{1}{l^4}$ and $g_3=140 \sum_{0\ne l\in\lat}\frac{1}{l^6}$
(in this example we will identify functions on $\mb{T}$ with $\lat$-invariant functions on $\mb{C}$).
The only poles of $\wp$ are at the lattice $\lat$. Hence if we choose $\rs=\mb{T}\setminus\{\lat\}$ then
\[\mc{O}_{\rs}=\mb{C}[\wp]\oplus\mb{C}[\wp]\wp',\quad S=\{\lat\}\] 
It is easy to see that $\wp(-z)=z$ and $\wp'(-z)=-\wp(z)$ which makes it also easy to construct an automorphic Lie algebra with reduction group $\rg=\cg{2}=\langle\gamma\rangle$ on the torus. Define $\sigma_{\rs}:\rg\mapsto\Aut{\rs}$ and $\sigma_{\mf{sl}_2(\mb{C})}:\rg\mapsto\Aut{\mf{sl}_2(\mb{C})}$ by \[\sigma_{\rs}(\gamma)z=-z,\quad \sigma_{\mf{sl}_2(\mb{C})}(\gamma)\begin{pmatrix}a&b\\c&-a\end{pmatrix}=\begin{pmatrix}a&-b\\-c&-a\end{pmatrix}.\]
Then we can write down a basis for the automorphic Lie algebra. Superscripts $\pm$ indicate the $\pm1$ eigenspaces of $\gamma$.
\begin{align*}
\mf{A}&=(\mf{sl}_2(\mb{C})\otimes \mc{O}_{\rs})^\rg
\\&=\mf{sl}_2(\mb{C})^+\otimes \mc{O}_{\rs}^+\oplus\mf{sl}_2(\mb{C})^-\otimes \mc{O}_{\rs}^-
\\&=\mb{C}\begin{pmatrix}1&0\\0&-1\end{pmatrix}\otimes \mb{C}[\wp]\oplus \mb{C}\left\langle\begin{pmatrix}0&1\\0&0\end{pmatrix},\begin{pmatrix}0&0\\1&0\end{pmatrix}\right\rangle\otimes \mb{C}[\wp]\wp'
\\&=\mb{C}\left\langle\begin{pmatrix}1&0\\0&-1\end{pmatrix}\otimes\wp^j,\,\begin{pmatrix}0&1\\0&0\end{pmatrix}\otimes\wp^j\wp',\,\begin{pmatrix}0&0\\1&0\end{pmatrix}\otimes\wp^j\wp'\,:\,j\ge0\right\rangle
\end{align*}
We pick a point $\genrs_0=l_1/2$ with nontrivial stabiliser. Define $\tilde{\wp}=\wp-\wp(l_1/2)$. The displayed formula above remains true if $\wp$ is replaced by $\tilde{\wp}$. Now $\wp'$ has a zero at $\genrs_0$ of order $1$ (due to antisymmetry) and therefore $\tilde{\wp}$ has a zero at $\genrs_0$ of order $2$. From this point it is easy to see that, for $j\ge0$,
\begin{align*}
&\mf{I}_{\genrs_0,2j}/\mf{I}_{\genrs_0,2j+1}=\mb{C}\left\langle\begin{pmatrix}1&0\\0&-1\end{pmatrix}\otimes\tilde{\wp}^j+\mf{I}_{\genrs_0,2j+1}\right\rangle,
\\&\mf{I}_{\genrs_0,2j+1}/\mf{I}_{\genrs_0,2j+2}=\mb{C}\left\langle\begin{pmatrix}0&1\\0&0\end{pmatrix}\otimes\tilde{\wp}^j\wp'+\mf{I}_{\genrs_0,2j+2},\,\begin{pmatrix}0&0\\1&0\end{pmatrix}\otimes\tilde{\wp}^j\wp'+\mf{I}_{\genrs_0,2j+2}\right\rangle,
\end{align*}
which exemplifies Theorem \ref{thm:composition factors}.
In particular, the quotient $\mf{A}/\mf{I}_{\genrs_0,m}$ of the automorphic Lie algebra has basis
\begin{align*}
&h^j=\begin{pmatrix}1&0\\0&-1\end{pmatrix}\otimes \tilde{\wp}^j+\mf{I}_{\genrs_0,m},\quad 2j<m
\\&e^j=\begin{pmatrix}0&1\\0&0\end{pmatrix}\otimes \tilde{\wp}^j\wp'+\mf{I}_{\genrs_0,m},\quad 2j+1<m
\\&f^j=\begin{pmatrix}0&0\\1&0\end{pmatrix}\otimes \tilde{\wp}^j\wp'+\mf{I}_{\genrs_0,m},\quad 2j+1<m
\end{align*}
and structure constants
\begin{align*}
&[h^i,e^j]=\left\{\begin{array}{ll} 2e^{i+j}&\text{if }i+j+1<m\\0&\text{otherwise}\end{array}\right.
\\&[h^i,f^j]=\left\{\begin{array}{ll} 2f^{i+j}&\text{if }i+j+1<m\\0&\text{otherwise}\end{array}\right.
\\&[e^i,f^j]=\left\{\begin{array}{ll} (4\wp^3-g_2 \wp- g_3)h^{i+j}&\text{if }i+j+2<m\\0&\text{otherwise.}\end{array}\right.
\end{align*}
This exemplifies Corollary \ref{cor:twisted truncated current algebra}, where one can make the isomorphism to the twisted truncated current algebra explicit by sending $\tilde{\wp}$ to $z_0^2$ and $\wp'$ to $z_0$ and carrying the group action to $\gamma z_0=-z_0$. 
\end{exam}

\section{The Local Structure of Automorphic Lie Algebras}
\label{sec:local}

In this section we specialise the Lie algebra $\mf{g}$ to be simple and root systems will enter the discussion. The twisted truncated current algebras can be described in root cohomological terms \cite{knibbeler2020cohomology} using the theory of automorphisms of finite order (a.k.a. torsions) of simple Lie algebras.

Assume that $\mf{g}$ is a simple Lie algebra of type ${X}_N$.
The eigenspace decomposition $\mf{g}=\bigoplus_{\bar{m}\in\zn{\nu_0}}\mf{g}_{\bar{m}}$ with respect a torsion $\gamma_0$ can be described using Kac coordinates, which we briefly explain here. For a more complete treatment we refer to \cite{kac1990infinite} where torsions are described in order to realise the affine Kac-Moody algebras concretely, and to \cite{reeder2010torsion} where the classification of torsions \cite{kac1969automorphisms} is achieved using only finite dimensional methods extending the reasoning of Cartan used to classify inner torsions.

Any torsion is of the form $\gamma_0=\mu g$ with $\mu$ an automorphism of $\mf{g}$ induced by an automorphism of its Dynkin diagram of order $r$, defining a second grading \[\mf{g}=\bigoplus_{\bar{l}\in\zn{r}}\mf{g}^{\bar{l}}\] and $g$ an inner torsion of the form $\exp \ad(x)$ with $x$ in a CSA of $\mf{g}^{\bar{0}}$ \cite[Prop. 8.1]{kac1990infinite}. This is explained by the following construction (see also Example \ref{ex:sl3}). The space $\mf{g}^{\gamma_0}$ contains a regular element $x'$. It can be scaled so that $\ad(x')$ has integer eigenvalues. Choose the centraliser of $x'$ in $\mf{g}$ as Cartan subalgebra of $\mf{g}$. This also fixes roots of $\mf{g}$. Define a root $\alpha$ to be positive when $\alpha(x')$ is positive. Then we have simple roots and corresponding weight vectors $E'_1,\ldots,E'_N\in\mf{g}$.

Since $\ad(x') \gamma_0=\gamma_0 \ad(x')$, the eigenspaces of $\ad(x')$, and in particular the CSA, are $\gamma_0$-invariant. Therefore $\gamma_0$ sends one weight space $\mf{g}_\alpha$ to another $\mf{g}_{\alpha\circ\gamma_0^{-1}}$. In that way it induces an automorphism of the root system preserving the positive roots, hence the simple roots, and hence a diagram automorphism $\bar{\mu}$, and an associated automorphism $\mu$ of $\mf{g}$. Say \[\gamma_0 E'_i = \zeta^{s_i}E'_{\bar{\mu}(i)}.\]
Then $g=\mu^{-1}\gamma_0$ sends $E'_i$ to $\zeta^{s_i}E'_{i}$ and is therefore of the form $g=\exp \ad(x)$ with $x$ in the CSA of $\mf{g}$. But we have more than that. The equation $\gamma_0 g^{-1} E'_i = E'_{\bar{\mu}(i)}$ shows that $s_i\equiv s_{\bar{\mu}(i)}$ modulo $\nu_0$. 
Therefore $g\mu=\mu g$ and $g=\exp \ad(x)$ with $x$ in the CSA of $\mf{g}^{\bar{0}}$.

Since the inner factor $g$ of $\gamma_0$ commutes with the diagram automorphism $\mu$ induced by $\gamma_0$, and with the action of the CSA of $\mf{g}^{\bar{0}}$, there is a basis diagonalising all three. As demonstrated in \cite[\S 8.3]{kac1990infinite}, such a basis is found by averaging the elements $E'_1,\ldots,E'_N$ over the action of $\mu$, and adding one more element to obtain a set of generators $E_0,\ldots,E_\ell$ of the Lie algebra $\mf{g}$ (where $\ell$ is the rank of $\mf{g}^{\bar{0}}$).
If $E_i$ has $\mu$-eigenvalue $(\zeta^{\frac{\nu_0}{r}})^l$ and  $g$-eigenvalue $\zeta^m$, we define $s_i$ to be the residue of $\frac{\nu_0}{r}l+m$ modulo $\nu_0$ (the exponent of the $\gamma_0$-eigenvalue). It turns out that the orbit of the base $(E_0,\ldots,E_\ell)$ under the action of the affine Weyl group contains a set of generators for which the integers $(s_0,\ldots,s_\ell)$ satisfy
\begin{equation}
\label{eq:kac coordinates}
\nu_0=r\sum_{i=0}^\ell a_i s_i.
\end{equation}
Kac showed that automorphisms of simple Lie algebras of order $m$, up to conjugacy, are in one-to-one correspondence with sequences $(s_0,\ldots,s_\ell)$ of relative prime nonnegative integers satisfying (\ref{eq:kac coordinates}), up to automorphisms of the Dynkin diagram of type $X_N^{(r)}$ \cite{kac1969automorphisms}. Nowadays such $(s_0,\ldots,s_\ell)$ are known as Kac coordinates (associated to $\gamma_0$).

\newcommand{\rer}{\mathrm{re}}
\newcommand{\imr}{\mathrm{im}}
We will formulate our main result in terms of Kac coordinates and a quotient of the root system of an affine Kac-Moody algebra. Let $P$ be the weight lattice of $\mf{g}^{\bar{0}}$. The nonzero weights of the $\mf{g}^{\bar{0}}$-representation $\mf{g}^{\bar{l}}$ will be denoted $\roots^{\bar{l}}$.  We define two subsets of the group $P\times\zn{r}$, namely \[\bar{\roots}_{\rer}=\bigcup_{\bar{l}\in\zn{r}}\roots^{\bar{l}}\times\{\bar{l}\},\quad \bar{\roots}_{\imr}=\{0\}\times \zn{r}\]
corresponding to the real and imaginary roots of the affine Kac-Moody algebra (united with zero). We endow their union $\bar{\roots}=\bar{\roots}_{\rer}\cup \bar{\roots}_{\imr}$ with the groupoid structure inherited from $P\times\zn{r}$. Then we have an isomorphism of groupoids
\[\bar{\roots}\cong(\roots(X_N^{(r)})\cup\{0\})/r\mb{Z}\delta\] where $\roots(X_N^{(r)})$ is the root system of the affine Kac-Moody algebra of type $X_N^{(r)}$ and $\delta$ the unique positive root $\delta=\sum_{i=0}^\ell a_i \alpha_i$ with coefficient vector $(a_0,\ldots,a_\ell)$ in the kernel of the Cartan matrix of type $X_N^{(r)}$. The coefficients $a_i$ are written in the Dynkin diagrams in \cite[Chap. 4]{kac1990infinite}. 
Notice that the special case of an inner automorphism corresponds to $r=1$, $\mf{g}^{\bar{0}}=\mf{g}$ and $\bar{\roots}\cong\roots(X_N)\cup\{0\}$, the root system of $\mf{g}$ united with zero.

The root systems $\roots(X_N^{(r)})$ and corresponding weight spaces are well understood.
The weight spaces corresponding to $\bar{\roots}_{\rer}$ are one-dimensional, the weight space of weight $(0,\bar{0})$ is $\ell$-dimensional, and weight spaces corresponding to $\bar{\roots}_{\imr}\setminus\{(0,\bar{0})\}$ are $\frac{N-\ell}{r-1}$-dimensional \cite[Cor. 8.3]{kac1990infinite}.
There is a set of simple weights $\sroots=\{\alpha_0,\ldots,\alpha_\ell\}$ in $\roots(X_N^{(r)})$ with defining property that positive element of $\roots(X_N^{(r)})$ is a linear combination of the $\alpha_i$ with nonnegative integer coefficients. Let $\bar{\sroots}$ denote its image in $\bar{\roots}$. We have $\bar{\sroots}\subset\bar{\roots}_\rer$.

We extend the Kac coordinates $s_i=\omega^1(\alpha)$ additively to a function on $\bar{\roots}$ \[\bar{\omega}^1:\bar{\roots}\rightarrow\zn{\nu_0}.\] We denote its residue by $\omega^1:\alpha\mapsto \res\bar{\omega}(\alpha)\in\{0,1,\ldots,\nu_0-1\}$. Moreover, for $\alpha,\beta,\alpha+\beta\in\bar{\roots}$ we define 
\begin{equation}
\label{eq:2-cocycle from eigenvalues}
\omega^2(\alpha,\beta)=\nu_0^{-1}(\omega^1(\beta)-\omega^1(\alpha+\beta)+\omega^1(\alpha))
\end{equation}
 and notice that $\omega^2(\alpha,\beta)\in\{0,1\}$.

When we combine the discussion of this section in with Corollary \ref{cor:twisted truncated current algebra} we obtain the following local structure theorem.
\begin{thm}[Local structure theorem]
\label{thm:local normal form}Let $\mf{A}$ be an automorphic Lie algebra on a punctured compact Riemann surface, where $\mf{g}$ is a simple Lie algebra of type $X_N$. 
Then there is a basis $\{A_{(\alpha,u)}\,:\,\alpha\in\bar{\roots},\,u=1,\ldots,d_\alpha\}$
of $\mf{g}$ diagonalising $\mu$, $g$, and $\ad\mf{h}^{\bar{0}}$, with structure constants $C_{(\alpha,u),(\beta,v)}^{(\alpha+\beta,w)}$ known from the affine Lie algebra $\mf{g}(X_N^{(r)})$. 
Hence the quotient $\mf{A}/\mf{I}_{\genrs_0,m}$ of the automorphic Lie algebra is isomorphic to the truncated twisted current algebra with basis
\[a_{(\alpha,u)}^i=A_{(\alpha,u)} \otimes z^{i\nu_0+{\omega}^1(\alpha)}+(z^{m})\]
for $\alpha\in\bar{\roots},\,u\in\{1,\ldots,d_\alpha\}$, and $0\le i\nu_0+{\omega}^1(\alpha)<m$.
The Lie structure is given by 
\[{[}a_{(\alpha,u)}^i,a_{(\beta,v)}^j]=\delta\sum_{w=1}^{d_{\alpha+\beta}} C_{(\alpha,u),(\beta,v)}^{(\alpha+\beta,w)}\,a_{(\alpha+\beta,w)}^{i+j+\omega^2(\alpha,\beta)}\]
where $\delta=1$ if $(i+j+\omega^2(\alpha,\beta))\nu_0+\omega^1(\alpha+\beta)<m$ and $\delta=0$ otherwise.
\end{thm}
If $\rg$ acts on $\mf{g}$ by inner automorphisms then the local Lie structure is easier to write down. The basis of  $\mf{A}/\mf{I}_{\genrs_0,m}$ becomes $h_1^i,\ldots,h_N^i$ with $0\le i\nu_0<m$, and $x_\alpha^i$ with $\alpha$ ranging in the root system $\roots$ of $\mf{g}$ and $0\le i\nu_0+{\omega}^1(\alpha)<m$. The Lie brackets are ${[}h_{i'}^i,h_{j'}^j]=0$, ${[}h_{i'}^i,a_{\alpha}^j]=\alpha(H_{i'})\,a_{\alpha}^{i+j}$ and
\[{[}a_{\alpha}^i,a_{\beta}^j]=\delta\, \epsilon({\alpha,\beta})\,a_{\alpha+\beta}^{i+j+\omega^2(\alpha,\beta)}\]
where $\delta=1$ if $(i+j+\omega^2(\alpha,\beta))\nu_0+\omega^1(\alpha+\beta)<m$ and $\delta=0$ otherwise.
Here $\epsilon({\alpha,\beta})$ is the usual structure constant of a Chevalley basis of $\mf{g}$.

We end this section with an elaborate example of a dihedral group action on a simple Lie algebra involving inner automorphisms, diagram automorphisms, and a product of such. We discuss the associated $1$-chain $\omega^1$ and $2$-cocycle $\omega^2$ that define the Lie algebra structure of automorphic Lie algebras with this group action in the way described by Theorem \ref{thm:local normal form}.
\begin{exam}[$\rg=\dg{6}$, $\mf{g}=\mf{sl}_3$]
\label{ex:sl3}
Consider the Lie algebra $\mf{g}=\mf{sl}_3$ and its standard concretisation with Chevalley basis
\begin{align*}
&E'_0=\left(\begin{array}{rrr}
0 & 0 & 0 \\
0 & 0 & 0 \\
1 & 0 & 0
\end{array}\right)
E'_1=\left(\begin{array}{rrr}
0 & 1 & 0 \\
0 & 0 & 0 \\
0 & 0 & 0
\end{array}\right)
E'_2=\left(\begin{array}{rrr}
0 & 0 & 0 \\
0 & 0 & 1 \\
0 & 0 & 0
\end{array}\right)
\\&
F'_0=[E'_1,E'_2],\quad F'_1=[E'_2,E'_0],\quad F'_2=[E'_0,E'_1],
\\&H'_1=[E'_1,F'_1],\quad H'_2=[E'_2,F'_2]
.
\end{align*}
The group of automorphisms of the root system of $\mf{sl}_3$ is a dihedral group $\Gamma=\langle a,b,c\,|\,a^6=b^2=c^2=a b c=1\rangle$ with $12$ elements. It is generated by a reflection in an arbitrary root and the nontrivial automorphism corresponding to the symmetry of the Dynkin diagram (interchanging the two simple roots). We pick the reflection in the simple root $\alpha_1$ and denote the associated automorphism of $\mf{sl}_3$ by $b$:
\begin{align*}
&bE'_0=F'_2,\quad bE'_1=F'_1,\quad bE'_2=F'_0
.
\end{align*}
The automorphism of $\mf{sl}_3$ corresponding to the symmetry of the Dynkin diagram will be denoted by $c$: 
\begin{align*}
&cE'_0=-E'_0,\quad cE'_1=E'_2,\quad cE'_2=E'_1
.
\end{align*}
We set $a=cb$. Then $a^6=b^2=c^2=a b c=1$. 

If $\Gamma$ acts faithfully on $\overline{\mb{C}}$, then each nontrivial stabiliser subgroup is conjugate to either $\langle a \rangle$, $\langle b \rangle$ or $\langle c \rangle$.
Theorem \ref{thm:local normal form} tells us that, near a point $\genrs_0$ with such nontrivial stabiliser, the automorphic Lie algebra associated to $\rg$ is locally described by the eigenspace decompositions $\mf{g}=\bigoplus_{m\in\zn{\nu_0}}\mf{g}_{\bar{m}}$ with respect to $\langle a \rangle$, $\langle b \rangle$ or $\langle c \rangle$.
We will discuss the inner generator $b$ first, the Dynkin automorphism $c$ second, and the mixed automorphism $a$ at the end.

An example of a regular element of $\mf{g}^b=\mf{g}_{\bar{0}}$ is \[H_0=\left(\begin{array}{rrr}
\frac{1}{2} & \frac{1}{2} & 0 \\
\frac{1}{2} & \frac{1}{2} & 0 \\
0 & 0 & -1
\end{array}\right),\] which has eigenvalues $0$ and $\pm1$. The centraliser of $h_0$ in $\mf{g}$ equals $\mf{h}=\mb{C}C\oplus\mb{C}H_0$ with \[C=\left(\begin{array}{rrr}
\frac{1}{3} & -1 & 0 \\
-1 & \frac{1}{3} & 0 \\
0 & 0 & -\frac{2}{3}
\end{array}\right).\]
This Cartan subalgebra defines six root spaces. For the roots that are simple in the affine system we take
\[E_0=\left(\begin{array}{rrr}
0 & 0 & 1 \\
0 & 0 & 1 \\
0 & 0 & 0
\end{array}\right),\quad 
E_1=\left(\begin{array}{rrr}
\frac{1}{2} & \frac{1}{2} & 0 \\
-\frac{1}{2} & -\frac{1}{2} & 0 \\
0 & 0 & 0
\end{array}\right),\quad
E_2=\left(\begin{array}{rrr}
0 & 0 & 0 \\
0 & 0 & 0 \\
-\frac{1}{2} & \frac{1}{2} & 0
\end{array}\right).
\]
These generate the remaining basis elements $F_{i\Mod{3}}=[E_{i+1\Mod{3}},E_{i+2\Mod{3}}]$ and we set $H_i=[E_i,F_i]$, $i=0,1,2$. The scaling factors of $E_i$ are chosen such that the structure constants of the basis $\{E_i,H_i,F_i\}$ correspond to those of the standard basis $\{E'_i,H'_i,F'_i\}$. 
We have \[\mf{g}_{\bar{0}}=\mb{C}E_0\oplus \mf{h}\oplus \oplus\mb{C}F_0,\quad\mf{g}_{\bar{1}}=\mb{C}E_1\oplus \mb{C}E_2\oplus \mb{C}F_1\oplus \mb{C}F_2.\] This shows that $b$ has Kac coordinates $(s_0,s_1,s_2)=(0,1,1)$ (only determined up to symmetry of the affine Dynkin diagram). 
These Kac coordinates define the one-form $\omega^1$ on $\bar{\roots}\cong\roots(A_2)\cup\{0\}$
\[\omega^1(0)=0,\quad \omega^1(\pm\alpha_0)=0,\quad \omega^1(\pm\alpha_1)=1,\quad \omega^1(\pm\alpha_2)=1.\] Its boundary $\omega^2$ can be presented as a graph with vertices $\bar{\roots}$ and undirected edges $\{\alpha,\beta\}$ for $\omega^2(\alpha,\beta)\ne 0$
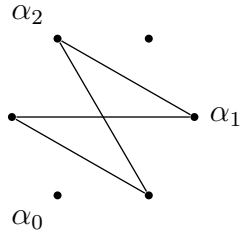
\newcommand{\sq}{0.866}
\tikzstyle{root}=[fill=black, shape=circle,minimum size=3 pt,inner sep=0]
\tikzstyle{cochain}=[line width=0.5pt]
\[\begin{tikzpicture}[scale=0.3]
  \path node at ( 0,0) [label=270: $ $] (zero) {$ $}	
	node at ( 4,0) [root,label=0: $\alpha_1 $] (one) {$ $}
  	node at ( 2,4*\sq) [root,label=60: $ $] (two) {$ $}
  	node at ( -2,4*\sq) [root,label=120: $\alpha_2 $] (three) {$ $}
	node at ( -4,0) [root,label=180: $ $] (four) {$ $}
	node at ( -2,-4*\sq) [root,label=240: $\alpha_0 $] (five) {$ $}
	node at ( 2,-4*\sq) [root,label=300: $ $] (six) {$ $};
	\draw [cochain](one) to node [sloped,above,near end]{} (four);
	\draw [cochain](three) to node [sloped,below,near end]{} (six);
	\draw [cochain](one) to node [sloped,above]{} (three);
	\draw [cochain](four) to node [sloped,below]{} (six);
\end{tikzpicture}\]
which gives a concise description of the bracket relations in $\mf{A}/\mf{I}_{\genrs_0,m}$ given in Theorem \ref{thm:local normal form}, an edge indicating one is added to the superindex, and no edge indicating nothing is added to the superindex.

Next we consider the automorphism $c$. Since it is induced by the symmetry of the Dynkin diagram, this case is also described in detail in \cite[\S 8.3]{kac1990infinite}. In the general construction we have $\mu=c$, $g=1$, so we can immediately proceed to the simultaneous diagonalisation of a CSA of $\mf{g}_0$ and $c$. The root system $\roots(A_2^{(2)})$ has a base $\bar{\sroots}=\{\alpha_0,\alpha_1\}$ 
\begin{align*}
&\bar{\roots}_\rer=\{\alpha_0,\alpha_1,\alpha_0+\alpha_1.3\alpha_0+\alpha_1, 4\alpha_0+\alpha_1, 3\alpha_0+2\alpha_1\},
\\& \bar{\roots}_\imr=\{0,\delta=2\alpha_0+\alpha_1\}
\end{align*}
We choose $H_0=2(H'_1+H'_2)$ to span the CSA, thereby fixing the simple root spaces to be the respective spans of $E_0=E'_1+E'_2$ and $E_1=E'_0$. The other basis elements are brackets of these two.
From $cE_0=E_0$ and $cE_1=-E_1$ we see that $c$ has Kac coordinates $(s_0,s_1)=(0,1)$ 
. 
Since $c=\mu$, their gradings coincide and we simply get $\bar{\omega}^1((\alpha,i))=i$, and from that follows $\omega^2((\alpha,i),(\beta,j))=1$ if $i=j=1$ and $\omega^2((\alpha,i),(\beta,j))=0$ otherwise.

We can present $\bar{\roots}$ graphically by making $\alpha$ the horizontal coordinate of $(\alpha,i)$ and $i$ the vertical\footnote{This picture needs a warning: it is only a representation of the groupoid structure of the root system. Not of the usual bilinear form defined on the dual of the Cartan subalgebra of the Kac-Moody algebra. This bilinear form is not positive definite ($(\delta,\delta)=0$).} (in this example $\alpha$ lives in the root lattice $\mf{sl}_2$ and $i$ in $\zn{2}$). We will write the value of $\omega^1$ above the weight. 
\[\begin{tabular}{c}\begin{tikzpicture}[scale=3*\dynkintablescale,baseline=(current bounding box.center), font=\dynkinfont,decoration={
    markings,
        mark=between positions 0.6 and 5 step 8mm with {\arrow{<}} }]
\path 
node at ( 0,0) [dynkinnode,label=90: $0 $,label=270: $0 $] (zero) {$ $}
node at ( -1,0) [dynkinnode,label=90: $0 $,label=270: $3\alpha_0+2\alpha_1 $] (one) {$ $}
node at ( 1,0) [dynkinnode,label=90: $0 $,label=270: $\alpha_0 $] (one) {$ $}
node at ( -2,1) [dynkinnode,label=90: $1 $,label=270: $\alpha_1 $] (one) {$ $}
node at ( -1,1) [dynkinnode,label=90: $1 $,label=270: $\alpha_0+\alpha_1 $] (one) {$ $}
node at (0,1) [dynkinnode,label=90: $1 $,label=270: $\delta $] (one) {$ $}
node at ( 1,1) [dynkinnode,label=90: $1 $,label=270: $3\alpha_0+\alpha_1 $] (one) {$ $}
node at ( 2,1) [dynkinnode,label=90: $1 $,label=270: $4\alpha_0+\alpha_1 $] (one) {$ $}
;
\end{tikzpicture}\end{tabular}\]

Lastly we turn to $a$. The algebra $\mf{g}^a$ is one-dimensional, spanned by the regular element 
\[H_0=
\frac{i\sqrt{3}}{3}\left(\begin{array}{rrr}
0 & 1 & 1 \\
-1 & 0 & -1 \\
-1 & 1 & 0
\end{array}\right)
.
\]
Its centraliser is a CSA defining simple weight spaces spanned by 
\begin{align*}
&E''_0=\left(\begin{array}{rrr}
-\frac{1}{216} i \, \sqrt{3} - \frac{1}{216} & \frac{1}{216} i \, \sqrt{3} - \frac{1}{216} & -\frac{1}{108} \\
\frac{1}{216} i \, \sqrt{3} - \frac{1}{216} & \frac{1}{108} & \frac{1}{216} i \, \sqrt{3} + \frac{1}{216} \\
-\frac{1}{108} & \frac{1}{216} i \, \sqrt{3} + \frac{1}{216} & \frac{1}{216} i \, \sqrt{3} - \frac{1}{216}
\end{array}\right),
\\&
E''_1=\left(\begin{array}{rrr}
\sqrt{3} - i & \sqrt{3} - i & -\sqrt{3} + i \\
2 i & 2 i & -2 i \\
\sqrt{3} + i & \sqrt{3} + i & -\sqrt{3} - i
\end{array}\right),
\\&
E''_2=\left(\begin{array}{rrr}
-\sqrt{3} + i & -2 i & -\sqrt{3} - i \\
-\sqrt{3} + i & -2 i & -\sqrt{3} - i \\
\sqrt{3} - i & 2 i & \sqrt{3} + i
\end{array}\right).
\end{align*}
We find that \[a E''_0=\zeta E''_0,\quad a E''_1=\zeta^4 E''_2,\quad a E''_2=\zeta^4 E''_1.\]
From there we readily see the factorisation in a diagram automorphism and an inner automorphism: 
\begin{align*}
&\mu E''_0=- E''_0,\quad \mu E''_1= E''_2,\quad \mu E''_2= E''_1.
\\& g E''_0=\zeta^4 E''_0,\quad g E''_1=\zeta^4 E''_1,\quad g E''_2= \zeta^4E''_2.
\end{align*}
We diagonalise $\mu$ and obtain the generators $E'''_0=E''_1+E''_2$ and $E'''_1=E''_0$ for the simple root spaces. Then $\gamma_0  E'''_0=\zeta^4 E'''_0$ and $\gamma_0 E'''_1=\zeta E'''_1$ and $\omega^1$ generated by $(s'_0,s'_1)=(4,1)$ is
\[\begin{tabular}{c}\begin{tikzpicture}[scale=3*\dynkintablescale,baseline=(current bounding box.center), font=\dynkinfont,decoration={
    markings,
        mark=between positions 0.6 and 5 step 8mm with {\arrow{<}} }]
\path 
node at ( 0,0) [dynkinnode,label=90: $0 $,label=270: $0 $] (zero) {$ $}
node at ( -1,0) [dynkinnode,label=90: $2 $,label=270: $3\alpha_0+2\alpha_1 $] (one) {$ $}
node at ( 1,0) [dynkinnode,label=90: $4 $,label=270: $\alpha_0 $] (one) {$ $}
node at ( -2,1) [dynkinnode,label=90: $1 $,label=270: $\alpha_1 $] (one) {$ $}
node at ( -1,1) [dynkinnode,label=90: $5 $,label=270: $\alpha_0+\alpha_1 $] (one) {$ $}
node at (0,1) [dynkinnode,label=90: $3 $,label=270: $\delta $] (one) {$ $}
node at ( 1,1) [dynkinnode,label=90: $1 $,label=270: $3\alpha_0+\alpha_1 $] (one) {$ $}
node at ( 2,1) [dynkinnode,label=90: $5 $,label=270: $4\alpha_0+\alpha_1 $] (one) {$ $}
;
\end{tikzpicture}\end{tabular}\]
This is sufficient to construct the structure constants of $\mf{A}/\mf{I}_{\genrs_0,m}$ using Theorem \ref{thm:local normal form}. However, $(s'_0,s'_1)=(4,1)$ are no Kac coordinates since they do not satisfy (\ref{eq:kac coordinates}). We can apply an element of the Weyl group $W(A^{(2)}_2)=\langle \sigma_0,\sigma_1\rangle$ in order to obtain Kac coordinates. In this case we find that $\sigma_1 \sigma_0$ takes $\omega^1$ to 
\[\begin{tabular}{c}\begin{tikzpicture}[scale=3*\dynkintablescale,baseline=(current bounding box.center), font=\dynkinfont,decoration={
    markings,
        mark=between positions 0.6 and 5 step 8mm with {\arrow{<}} }]
\path 
node at ( 0,0) [dynkinnode,label=90: $0 $,label=270: $0 $] (zero) {$ $}
node at ( -1,0) [dynkinnode,label=90: $5 $,label=270: $3\alpha_0+2\alpha_1 $] (one) {$ $}
node at ( 1,0) [dynkinnode,label=90: $1 $,label=270: $\alpha_0 $] (one) {$ $}
node at ( -2,1) [dynkinnode,label=90: $1 $,label=270: $\alpha_1 $] (one) {$ $}
node at ( -1,1) [dynkinnode,label=90: $2 $,label=270: $\alpha_0+\alpha_1 $] (one) {$ $}
node at (0,1) [dynkinnode,label=90: $3 $,label=270: $\delta $] (one) {$ $}
node at ( 1,1) [dynkinnode,label=90: $4 $,label=270: $3\alpha_0+\alpha_1 $] (one) {$ $}
node at ( 2,1) [dynkinnode,label=90: $5 $,label=270: $4\alpha_0+\alpha_1 $] (one) {$ $}
;
\end{tikzpicture}\end{tabular}\]
with Kac coordinates $(s_0,s_1)=(1,1)$. 
The associated Lie algebra generators are 
\begin{align*}
&E_0=\left(\begin{array}{rrr}
-6 \, \sqrt{3} + 6 i & -3 \, \sqrt{3} - 3 i & -6 i \\
-3 \, \sqrt{3} - 3 i & -12 i & -3 \, \sqrt{3} + 3 i \\
-6 i & -3 \, \sqrt{3} + 3 i & 6 \, \sqrt{3} + 6 i
\end{array}\right)
,
\\&E_1=\left(\begin{array}{rrr}
-\frac{1}{648} \, \sqrt{3} {\left(\sqrt{3} + 3 i\right)} & -\frac{1}{648} \, \sqrt{3} {\left(\sqrt{3} - 3 i\right)} & -\frac{1}{108} \\
-\frac{1}{648} \, \sqrt{3} {\left(\sqrt{3} - 3 i\right)} & \frac{1}{108} & \frac{1}{648} \, \sqrt{3} {\left(\sqrt{3} + 3 i\right)} \\
-\frac{1}{108} & \frac{1}{648} \, \sqrt{3} {\left(\sqrt{3} + 3 i\right)} & -\frac{1}{648} \, \sqrt{3} {\left(\sqrt{3} - 3 i\right)}
\end{array}\right).
\end{align*} 

All dihedral groups can be realized as automorphisms of the Riemann sphere and also as automorphisms of any complex torus. Consequently there are automorphic Lie algebras on punctured spheres and tori whose local Lie structure is described by the $2$-cocycles $\omega^2$ obtained in this example in the sense of Theorem \ref{thm:local normal form}.
\end{exam}

	\section{Further Directions} \label{sec:directions}
As we have seen, the representation theory of automorphic Lie algebras is incredibly rich and complicated. Due to the wild representation type of automorphic Lie algebras, one cannot hope to gain a complete understanding of their representation theory. However as a result of the description of the local structure theory of automorphic Lie algebras, there is promise to gain a much greater (albeit incomplete) understanding of their representation theory. There are numerous directions which one can pursue in this regard.

The finite-dimensional irreducible representations of automorphic Lie algebras have been classified in the general context of equivariant map algebras. However since the category $\fin \mfA$ of an automorphic Lie algebra $\mfA$ is not semisimple, this is very far from being the complete picture. Normally, one would look to the next largest class of representations that could be used to classify objects in $\fin \mfA$ --- namely, the indecomposable representations. But as we have seen, a classification of indecomposable representations of $\fin \mfA$ is largely hopeless. Thus, it makes sense to restrict ones attention to special subclasses of indecomposable representations.

\subsection{Bricks}
An example of a special class of indecomposable representations one may want to study is the class of bricks. These are representations with a trivial endomorphism algebra. Bricks are incredibly important examples of representations. As well as being fundamental objects in the category of representations they have many applications. Irreducible representations are an obvious example of bricks, but it is possible for there to exist many other bricks in an algebra that are not irreducible, and thus these form a wider class of representations to consider. Indeed, automorphic Lie algebras contain bricks that are not irreducible.

\begin{exam} \label{ex:Brick}
	Let $\mfA=\mathfrak{A}(\sltwo,\mathbb{X}, \Gamma, \sigma_{\sltwo},\sigma_\mbX)$ as in Example~\ref{ex:sl2Z5}. Consider the wild quotient Lie algebra
	\begin{equation*}
		\mfA'=\mfA / \mathfrak{I}_{0,3} \cong \left\langle 
		\begin{pmatrix}
			1	&	0 \\
			0	&	-1
		\end{pmatrix},
		\begin{pmatrix}
			0		&	0 \\
			\polyvar^2	&	0
		\end{pmatrix}
		\right\rangle.
	\end{equation*}
	Consider the representation $\rho = \ad_{\mfA'} \circ \ev_{\indrep{0}{3}} \in \fin \mfA$. Explicitly with respect to the basis given in Example~\ref{ex:sl2Z5}, we have
	\begin{equation*}
		\rho\left(
		\begin{pmatrix}
			1	&	0 \\
			0	&	-1
		\end{pmatrix}
		\right) =
		\begin{pmatrix}
			0	&	0 \\
			0	&	-2
		\end{pmatrix},\quad
		\rho\left(
		\begin{pmatrix}
			0		&	0 \\
			\polyvar^2	&	0
		\end{pmatrix}
		\right) =
		\begin{pmatrix}
			0	&	0 \\
			2	&	0
		\end{pmatrix}
	\end{equation*}
	and $\rho(a)=0$ for any other basis element $a$ of $\mfA$. It is easy to see that $\rho$ is not isomorphic to a point-evaluation representation (nor is it isomorphic to a point-evaluation representation tensored with a 1-dimensional representation of $\mfA$), and thus $\rho$ is not irreducible. However, a short calculation shows that $\End_{\mathbb{C}}(\rho) \cong \mathbb{C}$. Thus, $\rho$ is a brick.
\end{exam}

There are many examples of wild structures where the bricks are well-understood (see for example, \cite{Bodnarchuk2010} and references within), and thus it would be interesting to investigate what can be said about the bricks in automorphic Lie algebras.

\subsection{Uniserial Representations}
Uniserials are another important class of indecomposable representations. In some cases, it is possible to classify uniserial representations in a wild algebra -- an example being the particular solvable Lie algebras considered in \cite{Casati2017}. One can obtain many examples of uniserial representations of automorphic Lie algebras by evaluating with Jordan blocks.

A potential step in investigating the uniserial representations of automorphic Lie algebras is in understanding the space of extensions between representations. Formulas for the space of extensions between finite-dimensional irreducible representations of equivariant map algebras have been calculated in \cite{neher2015ext}. Of particular interest is when the space of extensions between irreducible representations is 1-dimensional, as this information can be used to construct uniserial representations and uniserial length subcategories of $\fin \mfA$ (see for example \cite{Eriksen2018}).

\printbibliography
\end{document}